\documentclass{amsart}
\usepackage{amsmath,amssymb,amsthm}
\usepackage{young}
\usepackage{booktabs}
\usepackage[usenames,dvipsnames]{color}
\usepackage{comment}
\usepackage{graphicx}
\usepackage{latexsym}
\usepackage[latin1]{inputenc}
\usepackage{mathtools}
\usepackage[shortlabels]{enumitem}
\usepackage{thmtools}
\usepackage{url}
\usepackage[all,cmtip]{xy}
\usepackage{tikz}
\usepackage{hyperref} %hyperref wants to be loaded last

\newcommand*\circled[1]{\tikz[baseline=(char.base)]{
    \node[shape=circle,draw,inner sep=1pt] (char) {#1};}}

\newcommand{\TransP}{\mathsf{Trans}}
\newcommand{\TP}{\mathsf{Tes}}
\newcommand{\FP}{\mathsf{Flow}}

\newcommand{\e}{\mathsf{e}}

\title{The polytope of Tesler matrices}

\author{Karola M\'esz\'aros}\address{Department of Mathematics, Cornell University, Ithaca NY}
\email{karola@math.cornell.edu}

\author{Alejandro H. Morales}\address{Department of Mathematics,
  University of California, Los Angeles, Los Angeles CA}
\email{ahmorales@math.ucla.edu}

\author{Brendon Rhoades}\address{Department of Mathematics, University
  of California, San Diego, La Jolla CA}
\email{bprhoades@math.ucsd.edu}

\thanks{M\'esz\'aros was partially supported by NSF Postdoctoral
  Research Fellowship DMS-1103933 and NSF Grant DMS-1501059. Morales was supported by a
  postdoctoral fellowship from CRM-ISM and LaCIM. Rhoades was partially supported by NSF grant DMS-1068861.}

\definecolor{darkgreen}{rgb}{0,0.7,0}
\definecolor{purplish}{rgb}{0.5,0,0.8}
\definecolor{cobalt}{rgb}{0.0, 0.28, 0.67}
\definecolor{auburn}{rgb}{0.43, 0.21, 0.1}

\hypersetup{
 breaklinks,
colorlinks,% more discreet colours for links
pdftitle={The polytope of Tesler matrices}
}

\newcommand{\RR}{\mathbb{R}}

\DeclareMathOperator{\dinv}{dinv}
\DeclareMathOperator{\area}{area}
\DeclareMathOperator{\bounce}{bounce}

\newcommand{\vol}{\operatorname{vol}}
\newcommand{\CT}{\operatorname{CT}}
\newcommand{\CTn}[1]{\CT_{x_{#1}}\cdots\CT_{x_1}}
\newcommand{\Res}{\operatorname{Res}}

\newcommand{\ZZ}{{\mathbb{Z}}}

\newcommand{\aaa}{{\bf a}}

\declaretheorem[numberwithin=section]{theorem}
\declaretheorem[numberlike=theorem]{lemma}
\declaretheorem[numberlike=theorem]{proposition}
\declaretheorem[numberlike=theorem]{corollary}

\declaretheorem[numberlike=theorem, style=definition]{remark}
\declaretheorem[numberlike=theorem, style=definition]{example}

\numberwithin{equation}{section} % requires package amsthm

\begin{document}

\begin{abstract} We introduce the Tesler polytope $\TP_n({\bf a})$, whose integer points are the
Tesler matrices of size $n$ with hook sums $a_1,a_2,\ldots,a_n \in
\mathbb{Z}_{\geq 0}$. 
We show that $\TP_n({\bf a})$ is a flow polytope and therefore 
the number of Tesler matrices is 
counted by the type $A_n$ Kostant partition function evaluated at
$(a_1,a_2,\ldots,a_n,-\sum_{i=1}^n a_i)$. We describe the faces of
this polytope in terms of ``Tesler tableaux'' and characterize when
the polytope is simple. We prove that the $h$-vector of $\TP_n({\bf
  a})$ when all $a_i>0$ is given by the Mahonian numbers and calculate the volume of $\TP_n(1,1,\ldots,1)$ to be a product of consecutive Catalan numbers multiplied by the number of standard Young tableaux of staircase shape.  
\end{abstract}

\maketitle
%\tableofcontents

%----------------------------------------------------------------
\section{Introduction}\label{sec:intro}
%----------------------------------------------------------------

Tesler matrices have played a major role in the works 
\cite{AGHRS}\cite{GHX}\cite{GoN}\cite{Hag}\cite{HRW}\cite{W} in the context of diagonal harmonics. We examine them from a
different perspective in this paper: we study the polytope, which we
call the Tesler polytope, consisting of upper triangular matrices with
nonnegative real entries with the same restriction as Tesler matrices
on the hook sums: sum of the elements of a row minus the sum of the elements of a
column. Then the integer points of this polytope are all Tesler
matrices of given hook sums.  We show that these polytopes are
flow polytopes and are faces of transportation polytopes. We
characterize the Tesler polytopes with nonnegative hook sums that are simple
and we calculate their $h$-vectors. If the hook sums are all $1$ the volume is the product of consecutive Catalan numbers multiplied by the number of standard Young tableaux of staircase shape. This result raises the question of the Tesler polytope's connection to the Chan-Robbins-Yuen polytope,  a flow polytope whose volume is the  product of consecutive Catalan numbers. 

We now proceed to give the necessary definitions and state our main results. This section is broken down into three subsections for 
ease of reading: introduction to Tesler matrices and polytopes,
introduction to flow polytopes and transportation polytopes, 
and our main results regarding Tesler polytopes. Section~\ref{h-section} and
Section~\ref{volume-section} are independent of each other, the first one is about
the face structure and the other is about the volume of Tesler
polytopes. Finally, in Section~\ref{final-remarks} we discuss some
final remarks and questions.

\subsection{Tesler matrices and polytopes} Let
$\mathbb{U}_n(\mathbb{R}_{\geq 0})$ be the set of $n \times n$ upper triangular
matrices with nonnegative real entries. The $k^\textsuperscript{th}$
{\bf hook sum} of a matrix $(x_{i,j})$ in
$\mathbb{U}_n(\mathbb{R}_{\geq 0})$ is the sum of all the elements of the $k\textsuperscript{th}$ row minus the
  sum of the elements in the $k\textsuperscript{th}$ column excluding the term in
  the diagonal:
\[
x_{k,k}+x_{k,k+1}+\cdots + x_{k,n} - \left(x_{1,k}+x_{2,k}+\cdots+x_{k-1,k}\right)
\]
Given a length $n$ vector ${\bf a} = (a_1, a_2,\ldots,a_n) \in (\ZZ_{\geq 0})^n$ of nonnegative integers, the
  {\bf Tesler polytope} $\TP_n({\bf a)}$ with hook sums ${\bf a}$ is the set 
  of matrices in $\mathbb{U}_n(\mathbb{R}_{\geq 0})$ where the
  $k\textsuperscript{th}$ hook sum equals $a_k$, for $k=1,\ldots,n$:
\[
\TP_n({\bf a}) = \{ (x_{i,j}) \in
\mathbb{U}_n(\mathbb{R}_{\geq 0}) \,:\,   x_{k,k}+\sum_{\mathclap{j=k+1}}^n
x_{k,j} - \sum_{i=1}^{k-1} x_{i,k} = a_k, 1\leq k \leq n \}.
\]

The lattice points of $\TP_n({\bf a})$ are called {\bf
    Tesler matrices} with hook sums ${\bf a}$. These are $n\times n$ upper triangular matrices $B=(b_{i,j})$ with nonnegative
integer entries such that  for $k=1,\ldots,n$, $b_{k,k} + \sum_{j=k+1}^n b_{k,j} - \sum_{i=1}^{k-1} b_{i,k} = a_k$.
The set and number of such matrices are denoted by $\mathcal{T}_n({\bf a})$ and
$T_n({\bf a})$ respectively. See
Figure~\ref{fig:teslerpolygraph} for an example of the seven Tesler
matrices in $\mathcal{T}_3(1,1,1)$. 

Tesler matrices
appeared recently in Haglund's study of diagonal harmonics \cite{Hag}
and their combinatorics and further properties were explored in
\cite{AGHRS}\cite{GHX}\cite{PL}\cite{W}. The flavor of the results obtained for Tesler matrices
in connection with diagonal harmonics is illustrated by the following
example. Let $\mathcal{H}(DH_n, q,t)$ denote the bigraded Hilbert series of the
space of diagonal harmonics $DH_n$. For more details regarding this polynomial in
$\mathbb{N}[q,t]$ we refer the reader to \cite[Ch. 5]{Hag2}. 

%For the definitions regarding the polynomial in
%$\mathbb{N}[q,t]$ in the left hand side of \eqref{eq:haglund} we refer the reader to \cite{Hag}.

\begin{figure}
\includegraphics{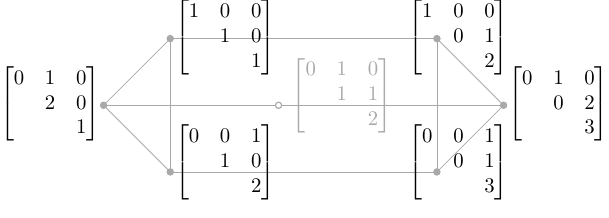}
\caption{The seven $3\times 3$ Tesler matrices with hook sums
  $(1,1,1)$. Six of them are vertices of the graph (depicted in gray) of the Tesler polytope $\TP_n(
    1,1,1)$.}
\label{fig:teslerpolygraph}
\end{figure}

\begin{example} \label{ex:tesler}
When ${\bf a} = {\bf 1}:=(1,1,\ldots,1) \in \mathbb{Z}^n$, Haglund \cite{Hag} 
showed that 
\begin{equation}\label{eq:haglund}
\mathcal{H}(DH_n,q,t) = 
\sum_{A \in \mathcal{T}_n(1,1,\ldots,1)} wt(A),
\end{equation}
where 
\begin{equation} \label{eq:Hwt}
wt(A) =\frac{1}{(-M)^n}\prod_{i,j \,:\, a_{ij}>0} (-M)[a_{ij}]_{q,t},
\quad M = (1-q)(1-t), \quad  [b]_{q,t} = \frac{q^b-t^b}{q-t}.
\end{equation}
\end{example}

The starting point for our investigation is the observation stated in the next lemma. 

\begin{lemma} \label{lem:flow} The Tesler polytope $\TP_n({\bf a})$ is 
a flow polytope $\FP_n({\bf a})$,
\begin{equation} \label{eq:TesPoly-to-FlowPoly}
\TP_n({\bf a}) \cong  \FP_n({\bf a}).
\end{equation}
\end{lemma}

We now define flow polytopes to make Lemma \ref{lem:flow} clear. For an illustration of the correspondence of polytopes in Lemma \ref{lem:flow} see  Figure~\ref{fig:exAll}. 

\begin{figure}
\includegraphics{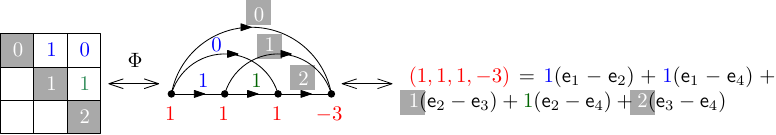}
\caption{Correspondence between a $3\times 3$ Tesler matrix with hook sums
  $(1,1,1)$, an integer flow in the complete graph $k_4$ and a vector
  partition of $(1,1,1,-3)$ into $\e_i-\e_j$ $1\leq i<j\leq 4$.}
\label{fig:exAll}
\end{figure}

\subsection{Flow polytopes} Given ${\bf a} = (a_1,a_2,\ldots,a_n)$, let
  $\FP_n({\bf a})$ be the {\bf flow polytope} of the
  complete graph $k_{n+1}$ with netflow $a_i$ on vertex $i$ for $i=1,\ldots,n$ and the netflow on vertex $n+1$ is  $-\sum_{i=1}^n
  a_i$. This polytope is the set of  functions  $f:E\to
  \mathbb{R}_{\geq 0}$, called flows,  from the edge set $E=\{(i,j) \,:\, 1\leq i<j\leq
  n+1\}$ of $k_{n+1}$ to the set
  of nonnegative real numbers such that for $k=1,\ldots,n$,
  $\sum_{j>s} f(k,j) - \sum_{i<k} f(i,k) = a_k$. This forces
  $\sum_{i=1}^n f(i,n+1) = \sum_{i=1}^n a_i$. We can write 
  $\FP_n({\bf a})=\{{\bf x} \in \mathbb{R}_{\geq 0}^{{n+1 \choose 2}}
  | A_{k_{n+1}} {\bf x}=({\bf a},  -\sum_{i=1}^n a_i)^T\}$, where
  $A_{k_{n+1}}$ is the matrix with columns $\e_i-\e_j$ for
  each edge $(i,j)$ of $k_{n+1}$, $1\leq i<j\leq n+1$. It is then evident that  the vertices of $\FP_n({\bf a})$ are integral, since $A_{k_{n+1}}$ is unimodular. 

\begin{proof}[Proof of Lemma~\ref{lem:flow}]  
Let $\Phi: \TP_n({\bf a}) \to  \FP_n({\bf a})$ defined by
$\Phi:X=(x_{i,j})\mapsto f_X$ where $f_X(i,j)=\begin{cases} x_{i,j}  \text{ if
  } j \neq n+1\\
x_{i,i}\text{ if } j=n+1
\end{cases}$. The map $\Phi$ is a linear transformation that simply permutes the coordinates
of $\TP_n({\bf a})$. Therefore the determinant of $\Phi$ is $\pm1$ and it
follows that $\Phi$ is a volume preserving bijection between the
polytopes.
\end{proof} 

The type
$A_n$ {\bf Kostant
partition function} $K_{A_n}({\bf a'})$ is the number of ways of writing
${\bf a'}:=({\bf a},  -\sum_{i=1}^n a_i)$ as an $\mathbb{N}$-combination of the type $A_n$ positive
roots $\e_i-\e_j$, $1\leq i<j\leq n+1$ without regard to order. Kostant partition functions are very useful in representation theory for
calculations of weight multiplicities and tensor product
multiplicities. The value $K_{A_n}({\bf a'})$ is
also the number of lattice points of the polytope $\FP_n({\bf
  a})$, i.e. integral flows in the complete graph $k_{n+1}$ with
netflow $a_i$ on vertex $i$ (see Figure \ref{fig:exAll} for an example). Thus the following lemma is immediate from Lemma \ref{lem:flow}.

\begin{lemma} \label{lem:tes} The number of Tesler matrices with hook sums $(a_1,a_2,\ldots,a_n)$ is given by the value   of the Kostant partition function at $(a_1,\ldots,a_n,-\sum_{i=1}^n a_i)$,
\begin{equation} \label{eq:Tes-to-Kpf}
T_n({\bf a}) = K_{A_n}({\bf a'}).
\end{equation}
\end{lemma}

In the next example we include a brief discussion of another flow polytope of the complete graph, namely,  $\FP_n({1, 0, \ldots, 0})$.
  
\begin{example} \label{ex:cry}
The polytope $\FP_{n}(1,0,\ldots,0)$ is known as the {\bf
  Chan-Robbins-Yuen polytope}. It has dimension $\binom{n}{2}$ and
$2^{n-1}$ vertices. Stanley-Postnikov (unpublished), and Baldoni-Vergne
\cite{BV2,BVMorris} proved that the normalized volume of this polytope is given by a value of
the Kostant partition function (see \eqref{eq:vol})
\[
\vol \FP_{n}(1,0,\ldots,0) =
K_{A_{n-1}}(0,1,2,\ldots,n-2,-{\textstyle \binom{n-1}{2}}).
\]
Then Zeilberger \cite{Z} used a variant of the Morris constant term identity \cite{WM}
to compute this value of the Kostant partiton
function as the product of the first $n-2$ Catalan numbers,
proving a conjecture of Chan, Robbins and Yuen \cite{CR,CRY}.
\begin{equation} \label{eq:z}
 K_{A_{n-1}}(0,1,2,\ldots,n-2,-{\textstyle \binom{n-1}{2}}) = \prod_{i=0}^{n-2} \frac{1}{i+1}\binom{2i}{i}.
\end{equation}
\end{example}
 
A Tesler polytope or flow
polytope is itself a
face of a well known kind of polytope called a transportation polytope
which we define next.

\subsection{Transportation polytopes} Given a vector ${\bf
  s}=(s_1,s_2,\ldots,s_n)$ of nonnegative integers, the {\bf transportation polytope}\footnote{In the literature transportation polytopes
  are more general \cite{KW}. The matrices can be rectangular and the $i\textsuperscript{th}$ row sum
  and the $i\textsuperscript{th}$ column sum can differ.}
$\TransP_n({\bf s})$ is the set of
all $n\times n$ matrices $M=(m_{i,j})$ with nonnegative real entries
whose $i\textsuperscript{th}$ row and $i\textsuperscript{th}$ column respectively sum to $s_i$, 
for $i=1,\ldots,n$. When all the $s_i$ equal one, the polytope
$\TransP_n(1,1,\ldots,1)$ is better known as the {\bf Birkhoff polytope}. 
Next we show that the flow polytope $\FP_n({\bf a})$ is isomorphic to a face of the
transportation polytope $\TransP_n(a_1,a_1+a_2,\ldots,\sum_{i=1}^na_i)$; see Figure~\ref{fig:tes2trans}. 

\begin{proposition} \label{prop:Tes2Trans}
For ${\bf a}=(a_1,\ldots,a_n) \in (\ZZ_{\geq 0})^n$
with $a_1>0$ we have that  
\begin{equation}
\label{eq:Tes2Trans}
\TP_n({\bf a}) \cong \{ (m_{i,j})\in
\TransP_n(a_1,a_1+a_2,\ldots,\sum_{i=1}^n a_i) \,:\, m_{i,j}=0 \text{ if
} i-j \geq 2\}.
\end{equation}
\end{proposition} 

For example, the Chan-Robbins
Yuen polytope $\TP_n(1,0,\ldots,0)$  is isomorphic to a face of the Birkhoff polytope
$\TransP_n(1,1,\ldots,1)$ \cite[Lemma 18]{BV2} and the Tesler polytope
$\TP_n(1,1,\ldots,1)$ is isomorphic to a face of the
transportation polytope $\TransP_n(1,2,\ldots,n)$. To prove the proposition we need the following characterization of the
facets of transportation polytopes \cite[Theorem 2]{KW} by Klee and Witzgall.

\begin{lemma} 
\cite{KW}
\label{lem:facetstrans}
Let ${\bf s}=(s_1,s_2,\ldots,s_n)$. The facets of
$\TransP_n({\bf s})$ are of the form $F_{i,j}({\bf s}):=\{M \in
\TransP_n({\bf s}) \,:\, m_{i,j}=0\}$ provided $s_i+s_j<\sum_{i=1}^n s_i$.
\end{lemma}

\begin{proof}[Proof of Proposition~\ref{prop:Tes2Trans}]
Fix ${\bf s} = (a_1,a_1+a_2,\ldots,\sum_{i=1}^n a_i)$ and let $\mathsf{F}_n$ denote the set on the right-hand-side of
\eqref{eq:Tes2Trans}. We claim that $\mathsf{F}_n$ is a
face of $\TransP_n({\bf s})$. If $n=1,2$ then
$\mathsf{F}_n=\TransP_n({\bf a})$ so the claim follows.
For $n\geq 3$ we have that $\mathsf{F}_n = \bigcap_{i-j\geq 2} F_{i,j}({\bf s})$. Since each $s_i\geq a_1>0$ then
$s_i+s_j<\sum_{i=1}^n s_i$ and by Lemma~\ref{lem:facetstrans} each $F_{i,j}({\bf s})$ is a facet of
$\TransP_n({\bf s})$. Thus $\mathsf{F}_n$ is a face of this transportation
polytope settling the claim. 

Next,  we build an isomorphism between $\FP_n({\bf a})$ and
$\mathsf{F}_n$. Then the result will follow by Lemma~\ref{lem:flow}. Let
$\Psi:\FP_n({\bf a}) \to \mathsf{F}_n$ be defined by $\Psi:f \mapsto (m_{i,j})$
where $m_{i,j}=\begin{cases}
f(i,j+1) & \text{ if } 1\leq i\leq j\leq n\\
\sum_{t=1}^ja_t - \sum_{t=1}^{j} f(t,j+1)      &  \text{ if } j=i-1,\\
0        & \text{ if } i-j \geq 2.
\end{cases}$.  See Figure~\ref{fig:tes2trans} for an example. 

We check that $(m_{ij})=\Psi(f)$ is in $\TransP_n({\bf s})$. We have that  $m_{j+1,j}=\sum_{k=1}^ja_k -
\sum_{k=1}^{j}f(k,j+1) \geq 0$ since $\sum_{k=1}^{j}f(k,j+1)$ is at most the total flow introduced at vertices $1,2, \ldots, j$, which is  $\sum_{k=1}^ja_k$. Therefore all the entries of $\Psi(f)$
are nonnegative. By construction, for $k=1,\ldots,n-1$ the sum of the
$k\textsuperscript{th}$ column is $\sum_{i=1}^ka_i$. The sum of the
$n\textsuperscript{th}$ column equals the netflow on vertex $n+1$ of
the complete graph, $\sum_{i=1}^n m_{i,n} = \sum_{i=1}^n f(i,n+1) = \sum_{i=1}^n a_i$.
For the rows, the netflow on vertex $k$, for $k=1,\ldots,n$, is $a_k$
which implies that $\sum_{j=k}^n f(k,j+1)- \sum_{i=1}^{k-1}
f(i,k) = a_k$. Thus the $k\textsuperscript{th}$ row sum equals
\begin{align*}
m_{k,k-1} + \sum_{j=k}^n m_{k,j} &= \left(\sum_{i=1}^{k-1} a_i -
\sum_{i=1}^{k-1} f(i,k)\right) + \sum_{j=k}^n f(k,j+1)\\
&= \sum_{i=1}^{k-1} a_i + \left(\sum_{j=k}^n f(k,j+1)- \sum_{i=1}^{k-1}
f(i,k)\right)= \sum_{i=1}^{k-1} a_i+a_k.
\end{align*}
Therefore $\Psi(f)$ is in $\TransP_n({\bf s})$. By construction $m_{i,j}=0$ if $i-j \geq
2$, so $\Psi(f)$ is also in $\mathsf{F}_n$ and so $\Psi$ is well
defined. Finally, we leave to the reader to check that $\Psi$ is a bijection with
inverse $\Psi^{-1}:\mathsf{F}_n \to \FP_n({\bf
  a})$, $(m_{ij})\mapsto f$ where $f(i,j)=m_{i,j-1}$ for $1\leq
i<j\leq n+1$. 
\end{proof}

\begin{figure}
\includegraphics{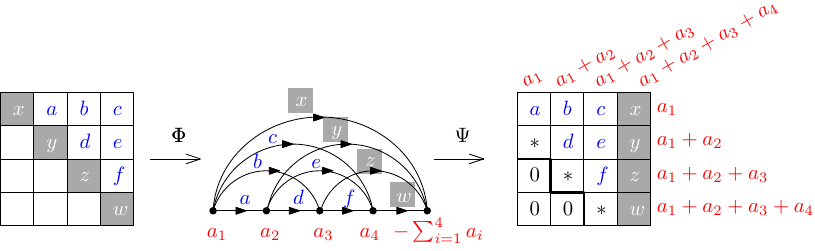}
\caption{Correspondence among: a point $X$ in the Tesler polytope
  $\TP_4({\bf a})$, a flow $f$ on the complete graph $k_5$ with netflow ${\bf
    a'}$, and a point $M=(m_{ij})$ of the transportation polytope
  $\TransP_4(a_1,a_1+a_2,\ldots,\sum_{i=1}^4 a_i)$ with
  $m_{31}=m_{41}=m_{42}=0$ (the entries marked $*$ are determined by
  the others).}
\label{fig:tes2trans}
\end{figure}

\subsection{The study of $\TP_n({\bf a})$} Examples~\ref{ex:tesler} and \ref{ex:cry}  served as our inspiration for studying   the Tesler polytope
$\TP_n({\bf a}) \cong \FP_n(\aaa)$. 
In Section~\ref{h-section} we prove that for any vector $\aaa \in (\ZZ_{\geq 0})^n$ of nonnegative integers, the polytope
$\TP_n(\aaa)$ has dimension ${n \choose 2}$
and at most $n!$ vertices, all of which are integral.  When $\aaa \in (\ZZ_{> 0})^n$ consists entirely of positive entries, we prove that 
$\TP_n(\aaa)$ has exactly $n!$ vertices.  In this case,
these vertices are the {\bf permutation Tesler
  matrices} of order $n$, which are the $n\times n$ Tesler matrices with at most one nonzero
entry in each row.  

Recall that if $P$ is a $d$-dimensional polytope, the {\bf $f$-vector} $f(P) = (f_0, f_1, \dots, f_d)$ of
$P$ is given by letting $f_i$ equal the number of faces of $P$ of dimension $i$.  The 
{\bf $f$-polynomial} of $P$ is the corresponding generating function
$\sum_{i = 0}^d f_i x^i$. A polytope $P$ is {\bf simple} if each of its
vertices is incident to $\dim(P)$ edges.  
If $P$ is a simple polytope, the {\bf $h$-polynomial} of $P$ is the polynomial $\sum_{i = 0}^d h_i x^i$
which is related to the $f$-polynomial of $P$ by the equation
$\sum_{i = 0}^d f_i (x-1)^i = \sum_{i = 0}^d h_i x^i$.  The coefficient sequence $(h_0, h_1, \dots, h_d)$
of the  $h$-polynomial of $P$ is called the {\bf $h$-vector} of $P$.

In Section~\ref{h-section} 
we characterize the vectors $\aaa \in (\ZZ_{\geq 0})^n$ for which  the
Tesler polytope $\TP_n(\aaa)$ is simple (Theorem~\ref{thm:charsimple}).  
In particular, we show that $\TP_n(\aaa)$ is simple whenever $\aaa \in (\ZZ_{> 0})^n$.  In this case,
the sum of its $h$-vector
entries is given by $\sum_{i = 0}^{\binom{n}{2}} h_i = f_0$.  Since
$\TP_n(\aaa)$ for ${\bf a} \in (\mathbb{Z}_{>0})^n$ has $n!$ vertices, this 
implies that $\sum_{i = 0}^{\binom{n}{2}} h_i = n!$.  One might expect that the $h$-polynomial
$\sum_{i = 0}^{\binom{n}{2}} h_i x^i$ of $\TP_n(\aaa)$ is the generating function of some interesting 
statistic on permutations.  Indeed, we show in Section~\ref{h-section} that the $h$-polynomial
of the Tesler polytope is the generating function for Coxeter length.

\begin{theorem} \label{thm:h} (Theorem~\ref{simple}, Corollary~\ref{mahonian})
Let $\aaa \in (\ZZ_{> 0})^n$ be a vector of positive integers.
The polytope $\TP_n(\aaa)$ is a simple polytope and its
$h$-vector 
is given by the
Mahonian numbers, that is, $h_i$ is the number of permutations of $\{1,2,\ldots,n\}$
with $i$ inversions. We have 
\[
\sum_{i=0}^{\binom{n}{2}} f_i (x-1)^i = \sum_{i=0}^{\binom{n}{2}} h_i
x^i = [n]!_x,
\]
where $[n]!_x = \prod_{i=1}^n (1+x+x^2+\cdots + x^{i-1})$ and the $f_i$ are the $f$-vector entries of $\TP_n({\bf 1})$.
\end{theorem}

Just as $\TP_n(1,0,\ldots,0)$, i.e. the
Chan-Robbins-Yuen polytope, $\FP_n(1,0,\ldots,0)$, has a product formula for its  normalized volume involving
Catalan numbers, so does the Tesler polytope $\TP_n({\bf
  1}) := \TP_n(1, 1, \dots, 1)$.  The following result is proven in
Section~\ref{volume-section} using a new iterated constant term
identity (Lemma~\ref{main:lemma}).

\begin{theorem} \label{thm:vol} (Corollary~\ref{cor:case11})
The normalized volume of the Tesler polytope
$\TP_n({\bf 1})$, or equivalently of the flow polytope
$\FP_n(1,1,\ldots,1)$ equals
\begin{align}
\vol\TP_n({\bf 1}) = \vol \FP_n(1,1,\ldots,1)
&=  \frac{\binom{n}{2}!\cdot
  2^{\binom{n}{2}}}{\prod_{i=1}^n i!} \notag \\
&= f^{(n-1,n-2,\ldots,1)} \cdot
\prod_{i=0}^{n-1} Cat(i), \label{eq:volformula}
\end{align}
where $Cat(i) = \frac{1}{i+1}\binom{2i}{i}$ is the $i\textsuperscript{th}$ Catalan number
and $f^{(n-1,n-2,\ldots,1)}$ is the number of Standard Young Tableaux of
staircase shape $(n-1,n-2,\ldots,1)$. 
\end{theorem}

\section{The face structure of $\TP_n(\aaa)$}
\label{h-section}

Let $\aaa \in (\ZZ_{\geq 0})^n$.
The aim of this section is to describe the face poset of $\TP_n(\aaa)$.  It will turn out that the combinatorial isomorphism
type of $\TP_n(\aaa)$ only depends on the positions of the zeros in the integer vector $\aaa$.  

Let $\mathrm{rstc}_n$ denote the reverse staircase of size $n$; the Ferrers diagram of $\mathrm{rstc}_4$ is shown below.

\begin{center}
\begin{Young}
   & & & \cr
 , & & & \cr
 ,  &, & & \cr
 ,   & ,& ,& 
\end{Young}
\end{center}

We use the  ``matrix coordinates" $\{(i, j) \,:\, 1 \leq i \leq j \leq n \}$ to describe the cells of $\mathrm{rstc}_n$.
An {\bf $\aaa$-Tesler tableau} $T$ is a $0,1$-filling of $\mathrm{rstc}_n$  which satisfies the following three conditions:
\begin{enumerate}
\item for $1 \leq i \leq n$, if $a_i > 0$, there is at least one $1$ in row $i$ of $T$,
\item for $1 \leq i < j \leq n$, if $T(i, j) = 1$, then there is at least one $1$ in row $j$ of $T$, and
\item for $1 \leq j \leq n$, if $a_j = 0$ and $T(i, j) = 0$ for all $1 \leq i < j$, then $T(j, k) = 0$ for all $j \leq k \leq n$.
\end{enumerate}
For example, if $n = 4$ and $\aaa = (7,0,3,0)$, then three $\aaa$-Tesler tableaux are shown below.  We write the entries of 
$\aaa$ in a column to the left of a given $\aaa$-Tesler tableau.

\begin{center}
\begin{Young}
,7 & 0  & 1 &1 & 1\cr
,0 & , &  0 & 0  & 1 \cr
,3 &  ,  &, & 1 &  1\cr
, 0 &,   & ,& ,& 1 
\end{Young}
\quad
\begin{Young}
,7&  1  &  0 &  1 & 0 \cr
,0& , & 0 & 0 & 0 \cr
 ,3& ,  &, &  0 & 1\cr
 ,0 &,   & ,& ,& 1 
\end{Young}
\quad
\begin{Young}
,7 & 1  & 1  & 1 &  0\cr
,0 & , & 1 &  1 & 0\cr
,3 & ,  &, & 1 & 0\cr
,0 & ,   & ,& ,& 0
\end{Young}
\end{center}

The {\bf dimension} $\dim(T)$ of an $\aaa$-Tesler tableau $T$ is $\sum_{i = 1}^n(r_i - 1)$, where 
\begin{equation*}
r_i = \begin{cases}
\text{the number of $1$'s in row $i$ of $T$} & \text{if row $i$ of $T$ is nonzero,} \\
1 & \text{if row $i$ of $T$ is zero.}
\end{cases}
\end{equation*}
From left to right, the dimensions of the tableaux shown above are 
$3, 1$, and $3$.  

Given two $\aaa$-Tesler tableaux $T_1$ and $T_2$, we write $T_1 \leq T_2$ to mean that for all $1 \leq i \leq j \leq n$ we have
$T_1(i,j) \leq T_2(i,j)$.  Moreover, we define a $0,1$-filling $\max(T_1, T_2)$ of $\mathrm{rstc}_n$ by
$\max(T_1, T_2)(i,j) = \max(T_1(i,j), T_2(i,j))$.  

We start with two lemmas on $\aaa$-Tesler tableaux.  Our first lemma states that any two zero-dimensional $\aaa$-Tesler tableaux
are componentwise incomparable.

\begin{lemma}
\label{zero-dimensional-distinct}
Let $\aaa \in (\ZZ_{\geq 0})^n$ and let $T_1$ and $T_2$ be two $\aaa$-Tesler tableaux with 
$\dim(T_1) = \dim(T_2) = 0$.  If $T_1 \leq T_2$, then $T_1 = T_2$.
\end{lemma}

\begin{proof}
Since $\dim(T_1) = \dim(T_2) = 0$, for all $1 \leq i \leq n$ we have that row $i$ of either $T_1$ or $T_2$
consists entirely of $0$'s, with the possible exception of a single $1$.  Since $T_1 \leq T_2$,
it is enough to show that if row $i$ of $T_2$ contains a $1$, then row $i$ of $T_1$ also contains a $1$.  To prove this,
we induct on $i$.  If $i = 1$, then row $1$ of $T_2$ contains a $1$ if and only if $a_1 > 0$, in which case
row $1$ of $T_1$ contains a $1$.  If $i > 1$, suppose that row $i$ of $T_2$ contains a $1$.  Then either 
$a_i > 0$ (in which case row $i$ of $T_1$ also contains a $1$) or $a_i = 0$ and there exists $i' < i$ such that
$T_2(i', i) = 1$.  But in the latter case we have that row $i'$ of
$T_1$ contains a $1$ by induction. This combined with the condition
$T_1 \leq T_2$ and the fact that $T_1$ and $T_2$ contain a unique $1$ in row $i'$ forces $T_1(i',i) = 1$.  Therefore,
row $i$ of $T_1$ contains a $1$.  We conclude that $T_1 = T_2$.
\end{proof}

Our next lemma states that the operation of componentwise maximum preserves the property of being an $\aaa$-Tesler tableau.

\begin{lemma}
\label{max-is-tesler}
Let $\aaa \in (\ZZ_{\geq 0})^n$ and let $T_1$ and $T_2$ be two $\aaa$-Tesler tableaux.
Then $T := \max(T_1, T_2)$ is also an $\aaa$-Tesler tableau.
\end{lemma}

\begin{proof}
If $a_i > 0$ for some $1 \leq i \leq n$, then row $i$ of $T$ is nonzero because row $i$ of $T_1$ is nonzero.  
If $1 \leq i < j \leq n$ and $T(i, j) = 1$, then either $T_1(i, j) = 1$ or $T_2(i, j) = 1$.  In turn, row $j$ of either $T_1$ or $T_2$ is nonzero,
forcing row $j$ of $T$ to be nonzero.  Finally, if $1 \leq j \leq n$, $a_j = 0$, and $T(i, j) = 0$ for all $1 \leq i < j$, then
$T_1(i, j) = T_2(i, j) = 0$ for all $1 \leq i < j$.  This means that row $j$ of $T_1$ and $T_2$ is zero, so row $j$ of $T$ is also zero.
\end{proof}

The analogue of Lemma~\ref{max-is-tesler} for $\min(T_1, T_2)$ is false; the componentwise minimum of two $\aaa$-Tesler 
tableaux is not in general an $\aaa$-Tesler tableau.
Faces of the Tesler polytope $\TP_n(\aaa)$ and $\aaa$-Tesler tableaux are related by taking supports.

\begin{lemma}
\label{faces-give-tableaux}
Let $\aaa \in (\ZZ_{\geq 0})^n$ and let $F$ be a face of the Tesler polytope $\TP_n(\aaa)$.  Define a function
$T: \mathrm{rstc}_n \rightarrow \{0, 1\}$ by $T(i, j) = 0$ if the coordinate equality $x_{i,j} = 0$ is satisfied on the face $F$ and
$T(i, j) = 1$ otherwise.  Then $T$ is an $\aaa$-Tesler tableau.
\end{lemma}

\begin{proof}
If $a_i > 0$ for some $1 \leq i \leq n$, we have $x_{i,i} + x_{i, i+1} + \cdots + x_{i,n} \geq a_i$ on the face $F$, so that 
row $i$ of $T$ is nonzero.  Suppose $T(i, j) = 1$ for some $1 \leq i < j \leq n$.  Then $x_{i,j} > 0$ holds for some point in $F$,
so that $x_{j,j} + x_{j,j+1} + \cdots + x_{j,n} \geq x_{i,j} > 0$ at that point.  In particular, row $j$ of $T$ is nonzero.
Finally, suppose that $a_j = 0$ and for all $1 \leq i < j$ we have $T(i, j) = 0$.  Then on the face $F$ we have
$x_{j,j} + x_{j,j+1} + \cdots + x_{j,n} = 0$, forcing $x_{j,j} = x_{j,j+1} = \cdots = x_{j,n} = 0$ on $F$.  This means that 
row $j$ of $T$ is zero.  
\end{proof}

Lemma~\ref{faces-give-tableaux} shows that every face $F$ of $\TP_n(\aaa)$ gives rise to an $\aaa$-Tesler tableaux $T$.
We denote by $\phi: F \mapsto T$ the corresponding map from faces of $\TP_n(\aaa)$ to $\aaa$-Tesler tableaux; we will see that
$\phi$ is a bijection.
We begin by showing that $\phi$ bijects vertices of $\TP_n(\aaa)$ with zero-dimensional $\aaa$-Tesler tableaux.

\begin{lemma}
\label{vertex-characterization}
Let $\aaa \in (\ZZ_{\geq 0})^n$.  The map $\phi$ bijects the vertices of $\TP_n(\aaa)$ with zero-dimensional
$\aaa$-Tesler tableaux.
\end{lemma}

\begin{proof}
Let $T$ be an $\aaa$-Tesler tableau with $\dim(T) = 0$.  Then $T$ contains at most a single $1$ in every row.  
There exists a unique point $B_T \in \TP_n(\aaa)$ such that the support of the matrix $B_T$ equals the set of nonzero entries
of $T$.  (Indeed, the vector $\aaa$ can be used to construct the matrix $B_T$ row by row, from top to bottom.)  By
Lemma~\ref{zero-dimensional-distinct}, we have that 
$B_{T_1} \neq B_{T_2}$ for distinct zero-dimensional $\aaa$-Tesler tableaux $T_1$ and $T_2$.
We argue that the set 
\begin{equation*}
\{ B_T \,:\, \text{$T$ an $\aaa$-Tesler tableau with $\dim(T) = 0$} \}
\end{equation*}
is precisely the set of vertices of $\TP_n(\aaa)$. Since this implies that $\phi(B_T) = T$, the lemma will follow.

To begin, we argue that $\TP_n(\aaa) =  \mathrm{conv}\{B_T \,:\, \dim(T) = 0\}$.  To facilitate this 
inductive argument, given any matrix
$B = (b_{i,j}) \in \TP_n(\aaa)$, define the {\em dimension} $\dim(B)$ to be $\dim(T)$, where $T$ is the $\aaa$-Tesler tableau 
whose entries are 
\begin{equation*}
T(i,j) = \begin{cases}
0 & b_{i,j} = 0 \\
1 & b_{i,j} \neq 0.
\end{cases}
\end{equation*}

Fix a matrix $B \in \TP_n(\aaa)$.  We want to show that $B \in \mathrm{conv} \{B_T \,:\, \dim(T) = 0 \}$.  We induct on $\dim(B)$.
If $\dim(B) = 0$, then $B = B_T$ for some $\aaa$-Tesler tableau $T$ with $\dim(T) = 0$ and the result follows, so 
assume $\dim(B) > 0$.  Since $\dim(B) > 0$, at least one row of $B$ has more than one positive entry.  Let 
$1 \leq i_0 \leq n-1$ be maximal such that row $i_0$ of $B$ has more than one positive entry.

For any $i_0 < j \leq n$ with $b_{i_0,j} > 0$, 
we define a subset $P_j = \{(p_1, q_1), (p_2, q_2), \dots \}$ 
of the matrix coordinates of $B$ (called the {\em positive path at $j$}) as follows.  Let $(p_1, q_1) = (i_0, j)$.
Given $(p_r, q_r) \in P_j$ with $p_r < n$, we define $(p_{r+1}, q_{r+1})$ by letting $p_{r+1} = q_r$
and letting $q_{r+1}$ be the column of the unique nonzero entry in row $q_r$ of $B$.
We also set $P_{i_0} = \{(i_0, j_0)\}$.  For example, if $\aaa = (4,3,1,1,1,2)$ and $B$ 
is the point in $\TP_6(\aaa)$ shown below, we have $i_0 = 2$ and
$P_3 = \{ (2,3), (3,5), (5,5) \},  P_4 = \{(2,4), (4,4)\},$ and $P_6 = \{(2,6), (6,6)\}$.
In general, for any distinct $j, j'$ we have $P_j \cap P_{j'} = \emptyset$.

\begin{center}
$B = \begin{bmatrix}
 0 & 2 & 0 & 1 & 1 & 0 \\
  & 0 & 2 & 2 & 0 & 1 \\
  &  & 0 & 0  & 3 & 0 \\
  &  &  & 4 & 0 &  0 \\
  &  &  &  & 5 & 0 \\
  &  &  &  &   & 3 
\end{bmatrix}$
\end{center}

Let $i_0 \leq j_0 < j_1 \leq n$ be such that  $c := b_{i_0, j_0}$ and $d := b_{i_0,j_1}$ are positive.  
We define two new upper triangular $n \times n$ matrices $B' = (b'_{i,j})$ and $B'' = (b''_{i,j})$ by the rules
\begin{equation}
b'_{i,j} = \begin{cases}
b_{i,j} + d & (i,j) \in P_{j_0} \\
b_{i,j} - d & (i,j) \in P_{j_1} \\
b_{i,j} & \text{otherwise}
\end{cases}
\end{equation}
and
\begin{equation}
b''_{i,j} = \begin{cases}
b_{i,j} - c & (i,j) \in P_{j_0} \\
b_{i,j} + c & (i,j) \in P_{j_1} \\
b_{i,j} & \text{otherwise.}
\end{cases}
\end{equation}
For example, if $B$ is as above with $i_0 = 2$, if we make the choices $j_0 = 3$ and $j_1 = 6$ the matrices
$B'$ and $B''$ are as follows.

\begin{center}
$B' = \begin{bmatrix}
 0 & 2 & 0 & 1 & 1 & 0 \\
  & 0 & 3 & 2 & 0 & 0 \\
  &  & 0 & 0  & 4 & 0 \\
  &  &  & 4 & 0 &  0 \\
  &  &  &  & 6 & 0 \\
  &  &  &   &   & 2
\end{bmatrix}$ \hspace{0.2in}
$B'' = \begin{bmatrix}
 0 & 2 & 0 & 1 & 1 & 0 \\
  & 0 & 0 & 2 & 0 & 3 \\
  &  & 0 & 0  & 1 & 0 \\
  &  &  & 4 & 0 &  0 \\
  &  &  &  & 3 & 0 \\
  &  &  &  &   & 5 
\end{bmatrix}$
\end{center}

It is straightforward to verify that both $B'$ and $B''$ lie in $\TP_n(\aaa)$.
Since $B'$ and $B''$ have one fewer positive entry
than $B$ in row $i_0$, we have $\dim(B') < \dim(B)$ and $\dim(B'') < \dim(B)$,
so that inductively 
$B' \in \mathrm{conv} \{B_T \,:\, \dim(T) = 0 \}$ and
$B'' \in \mathrm{conv} \{B_T \,:\, \dim(T) = 0 \}$.
Since $B = \frac{1}{c+d}(c B' + d B'')$, we conclude that
$B \in \mathrm{conv} \{B_T \,:\, \dim(T) = 0 \}$.

Since $\TP_n(\aaa) =  \mathrm{conv}\{B_T \,:\, \dim(T) = 0\}$, every vertex of $\TP_n(\aaa)$ is of the form $B_T$ for some 
$\aaa$-Tesler tableau $T$ with $\dim(T) = 0$.  We argue that every matrix $B_T$ is actually a vertex of $\TP_n(\aaa)$.
For otherwise, there would exist some $\aaa$-Tesler tableau $T$ with $\dim(T) = 0$ such that
\begin{equation*}
B_T = \sum_{\substack{\dim(T') = 0 \\ T' \neq T}} c_{T'} B_{T'},
\end{equation*}
for some $c_{T'} \geq 0$ with $\sum c_{T'} = 1$.  But this is impossible by Lemma~\ref{zero-dimensional-distinct}.  We conclude that
$B_T$ is a vertex of $\TP_n(\aaa)$.
\end{proof}

We are ready to characterize the face poset of $\TP_n(\aaa)$.

\begin{theorem}
\label{face-poset-characterization}
Let $\aaa \in (\ZZ_{\geq 0})^n$.  The support
map $\phi: F \mapsto T$ gives an isomorphism from the face poset of $\TP_n(\aaa)$ to the 
set of $\aaa$-Tesler tableaux, partially ordered by $\leq$.  For any face $F$, we have that $\dim(F) = \dim(\phi(F))$.
\end{theorem}

\begin{proof}
For any $\aaa$-Tesler tableau $T$, define a face $F(T) \subseteq \TP_n(\aaa)$ by letting $F(T)$ be the intersection of the hyperplanes
$\{ x_{i,j} = 0 \,:\, T(i, j) = 0\}$ within the ambient affine subspace 
\begin{equation*}
\bigcap_{i=1}^n \{ x_{i,i} + x_{i, i+1} + \cdots + x_{i,n} = a_i + x_{1,i} + \cdots + x_{i-1,i} \}
\end{equation*}
of $\{(x_{i,j}) \,:\, x_{i,j} \in \RR, 1 \leq i \leq j \leq n\}$.
It is evident that 
$\dim(F(T)) = \dim(T)$ and that $\phi(F(T)) = T$.  Moreover, we have that $T_1 \leq T_2$ if and only if $F(T_1) \subseteq F(T_2)$.
It therefore suffices to show that every face of $\TP_n(\aaa)$ is of the form $F(T)$ for some $\aaa$-Tesler tableau $T$.

Let $F$ be a face of $\TP_n(\aaa)$.  By Lemma~\ref{vertex-characterization}, there exist zero-dimensional $\aaa$-Tesler tableaux
$T_1, \dots, T_k$ such that $B_{T_1}, \dots, B_{T_k}$ are the vertices of $F$.  Let $T = \max(T_1, \dots, T_k)$.  By
Lemma~\ref{max-is-tesler} we have that $T$ is an $\aaa$-Tesler tableau.  It is clear that $F \subseteq F(T)$.  We argue that $F(T) \subseteq F$.
To see this, suppose that $1 \leq i \leq j \leq n$ and the defining hyperplane $x_{i,j} = 0$ of $\TP_n(\aaa)$ contains $F$.  Then in particular
we have that $x_{i,j} = 0$ contains $B_{T_1}, \dots, B_{T_k}$, so that 
$T_1(i, j) = \cdots = T_k(i, j) = 0$.  This means that $T(i, j) = 0$, so that $x_{i,j} = 0$ contains $F(T)$.  We conclude that $F = F(T)$.
\end{proof}

Given any vector $\aaa \in (\ZZ_{\geq 0})^n$, we let $\epsilon(\aaa) \in \{0, +\}^n$ be the associated {\bf signature}; for example,
$\epsilon(7,0,3,0) = (+, 0, +, 0)$.  Theorem~\ref{face-poset-characterization} implies that the combinatorial isomorphism type of
$\TP_n(\aaa)$ depends only on the signature $\epsilon(\aaa)$.

As a first application of Theorem~\ref{face-poset-characterization}, we
determine the dimension of $\TP_n(\aaa)$  and give an upper bound
on the number of its vertices.  When $\aaa \in \mathbb{Z}_{>0}^n$ the result about the dimensionality also follows from \cite{BV2}. Observe that if $a_1 = 0$, the first rows of the matrices in 
$\TP_n(\aaa)$ vanish and we have the identification $\TP_n(\aaa) = \TP_{n-1}(a_2, a_3, \dots, a_n)$.  We may therefore restrict to the 
case where $a_1 > 0$.

\begin{corollary}
\label{dimension}
Let $\aaa = (a_1, \dots, a_n) \in (\ZZ_{\geq 0})^n$ 
and assume $a_1 > 0$.  The polytope $\TP_n(\aaa)$ has dimension ${n \choose 2}$ and at most $n!$ vertices. 
Moreover, the polytope $\TP_n(\aaa)$ has exactly $n!$ vertices if and only if $a_2, a_3, \dots, a_{n-1} > 0$.
\end{corollary}
\begin{proof}
The claim about dimension follows from the fact that the mapping $T(i, j) = 1$ for $1 \leq i \leq j \leq n$ is an $\aaa$-Tesler tableau of
dimension ${n \choose 2}$ (since 
$a_1 > 0$).  

Recall that a {\bf file rook} is a rook which can attack horizontally, but not vertically
(see for example \cite[Definition 1]{BCHR}).  
There is an injective mapping 
from the set of zero-dimensional $\aaa$-Tesler tableaux to the set of maximal file rook placements on
$\mathrm{rstc}_n$ by placing a file rook in the position of every $1$ in $T$, together with a  file rook on the main diagonal
of any zero row of $T$.  Since there are $n!$ maximal file rook placements on $\mathrm{rstc}_n$, by
Theorem~\ref{face-poset-characterization}
we have that 
$\TP_n(\aaa)$ has at most $n!$ vertices.

If $a_2, a_3, \dots, a_{n-1} > 0$, then a zero-dimensional $\aaa$-Tesler tableau $T$ contains a unique $1$ in every row, with
the possible exception of row $n$ (which consists of a single cell).  Thus,
every maximal file rook placement on $\mathrm{rstc}_n$ arises from a zero-dimensional $\aaa$-Tesler tableau.  It follows that
$\TP_n(\aaa)$ has $n!$ vertices.  On the other hand, if $a_i = 0$ for some $1 < i < n$, then for any zero-dimensional $\aaa$-Tesler tableau
$T$ we have that $T(j,k) = 0$ for all $j < k$ implies $T(i,i) = 0$.  In terms of the corresponding
file rook placements, this means that if the 
file rooks in every row other than $i$ are on the main diagonal, then the file rook in row $i$ is also on the main diagonal.  
In particular, the mapping from zero-dimensional $\aaa$-Tesler tableaux to maximal file 
rook placements on $\mathrm{rstc}_n$ is not surjective and the polytope $\TP_n(\aaa)$ has $< n!$ vertices.
\end{proof}

Theorem~\ref{face-poset-characterization} can also be used to characterize when $\TP_n(\aaa)$ is a simple polytope.

\begin{theorem}
\label{thm:charsimple}
\label{simple}
Let $\aaa = (a_1, \dots, a_n) \in (\ZZ_{\geq 0})^n$ and let
$\epsilon(\aaa) = (\epsilon_1, \dots, \epsilon_n) \in \{0, +\}^n$ 
be the associated signature.  
Assume that $\epsilon_1 = +$.
The polytope $\TP_n(\aaa)$ is a simple polytope if and only if 
$n \leq 3$ or $\epsilon(\aaa)$ is one of $+^n, +^{n-1}0, +0+^{n-2}$ or $+0+^{n-3}0$.
\end{theorem}

\begin{proof}
When $n = 1$ the polytope $\TP_1(\aaa)$ is a single point.  When $n = 2$ the polytope $\TP_2(\aaa)$ is an interval.
When $n = 3$ the polytope $\TP_3(\aaa)$ is a $3$-simplex $\Delta_3$ if $\epsilon_2 = 0$ and
the triangular prism $\Delta_1 \times \Delta_2$ if $\epsilon_2 = +$.  In either case,  we have that $\TP_3(\aaa)$ is simple.

In general, the vertices of $\TP_n(\aaa)$ correspond to zero-dimensional $\aaa$-Tesler tableaux $T$.  
We may therefore speak of ``adjacent" zero-dimensional $\aaa$-Tesler tableaux $T_1$ and $T_2$ to mean that the corresponding 
vertices $B_{T_1}$ and $B_{T_2}$ are connected by an edge of $\TP_n(\aaa)$.
Given two distinct
$\aaa$-Tesler tableau $T_1, T_2$ with $\dim(T_1) = \dim(T_2) = 0$, by Theorem~\ref{face-poset-characterization}
we know that $T_1$ and $T_2$ are adjacent if and only if for all $1 \leq i \leq n$, row $i$ of $T_2$ can be obtained
from row $i$ of $T_1$ by 
\begin{enumerate}
\item leaving row $i$ of $T_1$ unchanged,
\item changing the unique $1$ in row $i$ of $T_1$ to a $0$,
\item changing a single $0$ in row $i$ to $T_1$ to a $1$ (if row $i$ of $T_1$ is a zero row), or
\item moving the unique $1$ in row $i$ of $T_1$ to a different position in row $i$.
\end{enumerate}
Moreover, the Operation $(4)$ must take place in precisely one row of $T_1$.

Given a fixed $\aaa$-Tesler tableau $T$ with $\dim(T) = 0$, 
we can replace the $0$'s in $T$ with entries in the set 
$\{ \circled{\,{\it i}\,} \,:\, i \in \ZZ_{\geq 0} \}$ to keep track of some of the adjacent zero-dimensional $\aaa$-Tesler tableaux.
In particular, we define a new filling $T^{\circ}$ of $\mathrm{rstc}_n$ 
using the alphabet $\{1, \circled{0}, \circled{1}, \circled{2}, \dots\}$ 
as follows.
\begin{itemize}
\item  If $T(i,j) = 1$, set $T^{\circ}(i,j) = 1$.
\item  If $T(i,j) = 0$ and row $i$ of $T$ is zero, then set $T^{\circ}(i, j) = \circled{0}$.
\item If $T(i,j) = 0$, row $i$ of $T$ is nonzero,  and row $j$ of $T$ is nonzero, then set 
$T^{\circ}(i, j) = \circled{1}$.
\item If $T(i,j) = 0$, row $i$ of $T$ is nonzero, and row $j$ of $T$ is zero, then set $T^{\circ}(i, j) = \circled{{\it j'}}$, where
$j' = n-j+1$ is the number of boxes in row $j$.
\end{itemize}
Observe that in the first case  we necessarily have $\epsilon_i = 0$ and in the  third case  we necessarily have $\epsilon_j = 0$.
For example, suppose $n = 5$ and $(\epsilon_1, \dots, \epsilon_5) = (+, 0, 0, 0, +)$.  Applying the above rules to the 
zero-dimensional $\aaa$-Tesler tableau $T$ shown below yields the given $T^{\circ}$.
\begin{center}
\begin{Young}
, \cr
, \cr
,$T = $ \cr
,\cr 
,\cr 
\end{Young}
\begin{Young}
,+ & 0 & 0 &0 &1 &0 \cr
,0 & , & 0& 0& 0&0 \cr
,0 &, &, & 0 & 0 & 0 \cr
,0 &, &, &, &  0 & 1 \cr
,+ &, &, &, &, & 1 \cr
\end{Young}
\begin{Young}
, \cr
, \cr
,$\leadsto$ \cr
,\cr 
,\cr 
\end{Young}
\begin{Young}
,+ & \circled{1} & \circled{4} & \circled{3} &1 & \circled{1} \cr
,0 & , & \circled{0} & \circled{0} & \circled{0} & \circled{0} \cr
,0 &, &, & \circled{0} & \circled{0} & \circled{0}  \cr
,0 &, &, &, &  \circled{1} & 1 \cr
,+ &, &, &, &, & 1 \cr
\end{Young}
\begin{Young}
, \cr
, \cr
,$= T^{\circ} $ \cr
,\cr 
,\cr 
\end{Young}
\end{center}

For any $\aaa$-Tesler tableau $T$ with $\dim(T) = 0$, we claim that the number of 
adjacent zero-dimensional $\aaa$-Tesler tableaux
is at least the sum of the circled entries in the associated tableau $T^{\circ}$.
For example, the number of adjacent tableaux in the case shown above is $\geq 1 + 4 + 3 + 1 + 1 = 10$.  To see this, observe
that for any adjacent zero-dimensional $\aaa$-Tesler tableau $T'$, there is precisely one row $i$ such that both $T$ and $T'$ contain a 
$1$ in row $i$, but this $1$ is in a different position (corresponding to Operation $(4)$ above).  
We can view $T'$ as being obtained from $T$ by moving this $1$ in row $i$,
and then possibly changing entries in lower rows (corresponding to Operations $(2)$ and $(3)$ above).
If this $1$ is moved to a position $(i, j)$ such that row $j$ of $T$ is zero, then one of the $j' = n-j+1$ $0$'s
 in row $j$ of $T'$ must be changed to a $1$.  In the example above, if the $1$ in position $(1,4)$ is moved to $(1,2)$, then one of the four
 $0$'s in positions $(2,2), (2, 3), (2, 4),$ and $(2,5)$ must be changed to a $1$, which corresponds to the circled $4$ in position
 $(1,2)$ of $T^{\circ}$.  
 We emphasize that this lower bound on the number of adjacent tableaux is not tight in general; for example, if we move the $1$
 in row $1$ in the above tableau from  $(1,4)$ to $(1,2)$ and change the $0$ in position $(2,3)$ to a $1$, then we must change one
 of the three $0$'s in row $3$ to a $1$, leading to more options for adjacent tableaux.
 In particular, the number of adjacent tableaux to the tableau $T$ shown above is 
 $> 10 = {5 \choose 2} = \dim(\TP_5(+,0,+,0,+,+))$ and the polytope $\TP_5(+,0,+,0,+,+)$ is not simple.
 
 Suppose that $n > 3$ and there exist indices $1 < i < j < n$ such that $\epsilon_i = +$ and $\epsilon_j = 0$.  We argue that 
 $\TP_n(\aaa)$ is not simple by exhibiting an $\aaa$-Tesler tableau $T$ such that $T$ has $> {n \choose 2} = \dim(\TP_n(\aaa))$ adjacent
 zero-dimensional $\aaa$-Tesler tableaux.  Indeed, let $T$ be the ``diagonal" $\aaa$-Tesler tableau defined by
 $T(k, \ell) = 0$ whenever $1 \leq k < \ell \leq n$, $T(i, i) = 1$ if $\epsilon_i = +$, and $T(i, i) = 0$ if $\epsilon_i = 0$.  Perform the above circling
 procedure to $T$ to get the tableau $T^{\circ}$; the example $\epsilon = (+, 0, +, 0, +, +)$ is shown below.
\begin{center}
\begin{Young}
,+ & 1 & 0 &0 &0 &0 & 0  \cr
,0 & , & 0& 0& 0&0 & 0 \cr
,+ &, &, & 1 & 0 & 0 & 0 \cr
,0 &, &, &, &  0 & 0 & 0  \cr
,+ &, &, &, &, & 1 & 0  \cr
,+ &, &, &, &, &, & 1 \cr
\end{Young}
\begin{Young}
, \cr
, \cr
,$\leadsto$ \cr
,\cr 
,\cr 
,\cr
\end{Young}
\begin{Young}
,+ &1 & \circled{5} & \circled{1} & \circled{3} & \circled{1} & \circled{1} \cr
,0 & , & \circled{0} & \circled{0} & \circled{0} & \circled{0}  & \circled{0} \cr
,+ &, &, & 1 & \circled{3} & \circled{1} & \circled{1} \cr
,0 &, &, &, &  \circled{0} & \circled{0} & \circled{0} \cr
,+ &, &, &, &, & 1 & \circled{1} \cr
,+ &, &, &, &, &, & 1 \cr
\end{Young}
\end{center}
We claim that the sum of the circled entries in row $1$ of $T^{\circ}$, plus the number of circled positive entries in the remaining rows 
of $T^{\circ}$,
equals ${n \choose 2}$.  
Indeed,
since $\epsilon_1 > 0$, we have the entry in position $(1, k)$ of $T^{\circ}$ is a positive circled number for $2 \leq k \leq n$.  
If $T^{\circ}(1,k) = \circled{1}$,
then row $k$ of $T$ is nonzero, so that row $k$ of $T^{\circ}$ consists of precisely one $1$, together with 
$n-k$ $\circled{1}$'s.
If $T^{\circ}(1,k) = \circled{$k'$}$ for some $k' > 1$, we must have that $k' = n-k+1$, $\epsilon_k = 0$, and
row $k$ of $T^{\circ}$ consists entirely of $\circled{0}$'s.
In either case, the circled entry in $T^{\circ}(1,k)$, plus the number of positive circled entries in row $k$ of $T^{\circ}$,
is one plus the number of boxes in row $k$ of $T^{\circ}$.
On the other hand, the entry in position $(i, j)$ of $T^{\circ}$
is a circled number $> 1$ because $\epsilon_j = 0$
and $j < n$.  This means that the sum of the circled entries is $> {n \choose 2}$, the tableau
$T$ has $> {n \choose 2}$ adjacent zero-dimensional tableaux, and the polytope
$\TP_n(\aaa)$ is not simple.

Suppose that $n > 3$ and $\epsilon$ has the form $\epsilon = + 0^i +^{n-i-1}$ for some $1 < i < n$.  Let $T$ be the 
``near-diagonal" zero-dimensional $\aaa$-Tesler tableau defined by $T(1,2) = T(2,2) = 1$, $T(j, j) = 1$ for $i < j \leq n$, and
$T(k, \ell) = 0$ otherwise.  Perform the above circling procedure to $T$ to get $T^{\circ}$; the case
$\epsilon = (+, 0, 0, 0, +, +)$ is shown below.
\begin{center}
\begin{Young}
,+ & 0 & 1 &0 &0 &0 & 0  \cr
,0 & , & 1& 0& 0&0 & 0 \cr
,0 &, &, & 0 & 0 & 0 & 0 \cr
,0 &, &, &, &  0 & 0 & 0  \cr
,+ &, &, &, &, & 1 & 0  \cr
,+ &, &, &, &, &, & 1 \cr
\end{Young}
\begin{Young}
, \cr
, \cr
,$\leadsto$ \cr
,\cr 
,\cr 
,\cr
\end{Young}
\begin{Young}
,+ & \circled{1} & 1 & \circled{4} & \circled{3} & \circled{1} & \circled{1} \cr
,0 & , &1 & \circled{4} & \circled{3} & \circled{1}  & \circled{1} \cr
,0 &, &, & 0 & \circled{0} & \circled{0} & \circled{0} \cr
,0 &, &, &, &  0 & \circled{0} & \circled{0} \cr
,+ &, &, &, &, & 1 & \circled{1} \cr
,+ &, &, &, &, &, & 1 \cr
\end{Young}
\end{center}
A similar argument as in the last paragraph shows that
the sum of the circled entries in row $1$ of $T^{\circ}$, plus the number of positive circled entries in the remaining rows of $T^{\circ}$, equals 
${n \choose 2}$.  On the other hand, since $1 < i < n$ and $n > 3$, at least one of the circled entries in row $2$ of $T^{\circ}$ is $> 1$.  
We conclude that the sum of all the circled entries is $> {n \choose 2}$, so that $\TP_n(\aaa)$ is not simple.

If $\epsilon_n = +$, let $\aaa' = (a_1, a_2, \dots, a_{n-1}, 0)$.  We claim that the polytopes $\TP_n(\aaa)$
and $\TP_n(\aaa')$ are affine isomorphic:  $\TP_n(\aaa) \cong \TP_n(\aaa')$.  Indeed, an isomorphism $B \mapsto B'$ is obtained
by subtracting $a_n$ from the $(n, n)$-entry of any matrix $B \in \TP_n(\aaa)$.  By this fact and the last two paragraphs, the polytope
$\TP_n(\aaa)$ is not simple unless $\epsilon(\aaa)$ has one of the four forms given in the statement of the theorem.   
Also by this fact, to complete the proof we need only show that $\TP_n(\aaa)$ is simple when $\epsilon(\aaa)$ has one of the two
forms $+^n$ or $+0+^{n-2}$.

If $\epsilon(\aaa) = +^n$, then any zero-dimensional $\aaa$-Tesler tableau has a unique $1$ in every row.
Given an $\aaa$-Tesler tableau $T$ with $\dim(T) = 0$, the tableaux adjacent to $T$ can be obtained by moving a single $1$ to 
a different position in its row.  There are $(n-1) + (n-2) + \cdots + 1 = {n \choose 2} = \dim(\TP_n(\aaa))$ ways to do this, so 
the polytope $\TP_n(\aaa)$ is simple.

If $\epsilon(\aaa) = +0+^{n-2}$, then any zero-dimensional $\aaa$-Tesler tableau $T$ has a unique $1$ in every row, with the possible
exception of row $2$.   In particular, row $2$ of $T$ contains a $1$ if and only if the $1$ in row $1$ of $T$ is in position $(1, 2)$.  
In either case, we see that $T$ is adjacent to precisely ${n \choose 2}$ tableaux, so that $\TP_n(\aaa)$ is simple.
\end{proof}

We now focus on the case of greatest representation theoretic interest in the context of diagonal harmonics: where $\epsilon(\aaa) = +^n$, so that 
every entry of $\aaa$ is a positive integer.  
The combinatorial
isomorphism type of $\TP_n(\aaa)$ is immediate from Theorem~\ref{face-poset-characterization}.  We denote 
by $\Delta_d$ the $d$-dimensional simplex in $\RR^{d+1}$ defined by
$\Delta_d := \{(x_1, \dots, x_{d+1}) \in \RR^{d+1} \,:\, x_1 + \cdots + x_{d+1} = 1, x_1 \geq 0, \dots, x_{d+1} \geq 0 \}$.

\begin{corollary}
\label{increasing-simplices}
Let $\aaa \in (\ZZ_{> 0})^n$ be a vector of positive integers.
The face poset of the Tesler polytope $\TP_n(\aaa)$ is isomorphic to
the face poset of the Cartesian product  of simplices
$\Delta_1 \times \Delta_2 \times \cdots \times \Delta_{n-1}$.
\end{corollary}

\begin{corollary}
\label{mahonian}
Let $\aaa \in (\ZZ_{> 0})^n$ be a vector of positive integers.
The $h$-polynomial of the Tesler polytope $\TP_n(\aaa)$ is the Mahonian distribution
\begin{equation*}
\sum_{i = 0}^{{n \choose 2}} h_i x^i = [n]!_x = (1+x)(1+x+x^2) \cdots (1+x+x^2 + \cdots + x^{n-1}).
\end{equation*}
\end{corollary}

\begin{proof}
We give two proofs of this result, one relying on Corollary~\ref{increasing-simplices} and one relying on generic linear forms.

{\it First proof:}  Let $P$ and $Q$ be arbitrary simple polytopes and let $P \times Q$ be their 
Cartesian product.  The polytope $P \times Q$ is simple and the $h$-polynomial of $P \times Q$ is the 
product of the $h$-polynomials of $P$ and $Q$.  
To see this,  observe that a typical $i$-dimensional face of
$P \times Q$ is given by the product of an $j$-dimensional face of $P$ and a $i-j$-dimensional face
of $Q$, for some $0 \leq j \leq i$.  Therefore,
 the $f$-vectors $f(P) = (f_0(P), f_1(P), \dots )$ and $f(Q) = (f_0(Q), f_1(Q), \dots)$
are related to the $f$-vector of the product $f(P \times Q)$ by 
$f_i(P \times Q) = \sum_{j = 0}^i f_i(P) f_{i-j}(Q)$.
The $h$-polynomials are therefore related by:
\begin{align*}
\sum_{i = 0}^{\dim(P) + \dim(Q)} h_i(P \times Q) x^i &= \sum_{i = 0}^{\dim(P) + \dim(Q)} f_i(P \times Q) (x-1)^i \\
&= \sum_{i = 0}^{\dim(P) + \dim(Q)} \left( \sum_{j = 0}^i f_j(P) (x-1)^j f_{i-j}(Q) (x-1)^{i-j} \right) \\
&= \left( \sum_{i = 0}^{\dim(P)} f_i(P) (x-1)^i \right) \left( \sum_{j = 0}^{\dim(Q)} f_j(Q) (x-1)^j \right),
\end{align*}
which equals the product of the $h$-polynomials of $P$ and $Q$. This multiplicative property of $h$-polynomials is surely well known, but the authors could not find a reference.

It remains to observe that the $h$-polynomial of the $d$-dimensional simplex $\Delta_d$ is given by
$\sum_{i = 0}^d h_i(\Delta_d) x^i = \sum_{i = 0}^d {d+1 \choose i+1}(x - 1)^i = 1 + x + \cdots + x^d$,
where we used the fact that $\Delta_d$ has ${d + 1 \choose i + 1}$ faces of dimension $i$.

{\it Second proof:}  Let $\lambda$ be any generic linear form on the vector space spanned by $\TP_n({\bf 1})$.
Then $\lambda$ induces an orientation on the $1$-skeleton of $\TP_n(\aaa)$ by requiring that the 
value of $\lambda$ increase along each oriented edge.  
It follows (see for example \cite[\textsection 8.3]{Ziegler})
that the $h$-vector entry $h_i(\TP_n(\aaa))$ equals the number of vertices in this oriented $1$-skeleton 
with outdegree $i$.

By Theorem~\ref{face-poset-characterization}, 
 the vertices of $\TP_n(\aaa)$ are the permutation Tesler matrices of size $n$ and
the edges of $\TP_n(\aaa)$ emanating from a fixed vertex correspond to changing the 
support of the corresponding permutation Tesler matrix of the vertex in exactly two positions 
belonging the the same row.
Let $\lambda$ be any linear form such that moving from one to another permutation Tesler matrix by shifting the support to the right in a single row corresponds to an increase in $\lambda$.
Then if the support of a permutation Tesler matrix is given by $\{ (i, b_i) \,:\, 1 \leq i \leq n \}$, its outdegree
in the orientation induced by $\lambda$ is $\sum_{i = 1}^n (n - b_i)$.  The corresponding generating 
function for outdegree is 
$\sum_{i = 0}^{{n \choose 2}} h_i(\TP_n(\aaa)) x^i = \prod_{i = 1}^n
\left(\sum_{a_i = i}^n x^{n - b_i} \right) = [n]!_x$.
\end{proof}

Corollaries~\ref{increasing-simplices} and \ref{mahonian} are also true for Tesler polytopes $\TP_n(\aaa)$, where
$\epsilon(a) = +^{n-1} 0$.
In light of Theorem~\ref{simple}, it is natural to ask for an analog to these results when 
$\epsilon(a)$ is of the form $+0+^{n-2}$ or $+0+^{n-3}0$.  Such an analog is provided by the following corollary.

\begin{corollary}
\label{other-simple-case}
Let $\aaa \in (\ZZ_{\geq 0})^n$ and assume that $\epsilon(\aaa)$ has one of the forms 
$+0+^{n-2}$ or $+0+^{n-3}0$.  Let $P$ be the quotient polytope $(\Delta_{n-2} \times \Delta_{n-1})/\sim$, where we declare
$(p, q) \sim (p', q)$ whenever $q \in \Delta_{n-1}$ belongs to the facet of $\Delta_{n-1}$ defined by $x_2 = 0$ and 
$p, p' \in \Delta_{n-2}$.

The face poset of the polytope $\TP_n(\aaa)$ is isomorphic to the face poset of the Cartesian product
$\Delta_1 \times \Delta_2 \times \cdots \Delta_{n-3} \times P$.  Moreover, we have that
$\TP_n(\aaa)$ has $2(n-1)!$ vertices and  $h$-polynomial $(1+x^{n-1})[n-1]!_x$.
\end{corollary}

\begin{proof} (Sketch.)
The second row of any $\aaa$-Tesler tableau $T$ is nonzero if and only if $T(1,2) = 1$.  All other rows of any $\aaa$-Tesler tableau
are nonzero.  By Theorem~\ref{face-poset-characterization}, we get the claimed Cartesian product decomposition of 
$\TP_n(\aaa)$.  The fact that $\TP_n(\aaa)$ has $2 (n-1)!$ vertices arises from the fact that the quotient polytope
$P$ has $2(n-1)$ vertices.  The fact that $\TP_n(\aaa)$ has $h$-polynomial $(1+x^{n-1})[n-1]!_x$ can be deduced from the 
multiplicative property of $h$-polynomials of the first proof of Corollary~\ref{mahonian} and the fact that
$P$ has $h$-polynomial $(1+x^{n-1})[n-1]_x$.
\end{proof}

\begin{remark}
All of the results of this section
 are still true when one considers the ``generalized" Tesler polytopes
 polytopes $\TP_n(\aaa)$ defined for real vectors $\aaa$; one simply replaces $(\ZZ_{\geq 0})^n$ and 
 $(\ZZ_{> 0})^n$ with $(\mathbb{R}_{\geq 0})^n$ and $(\mathbb{R}_{> 0})^n$ throughout.  The proofs are identical.
\end{remark}

\begin{remark}
When $\aaa \in (\ZZ_{> 0})^n$ is a vector of positive integers, Theorem~\ref{face-poset-characterization}
can be deduced from results of Hille \cite{hill}.  In particular, if $Q$ denotes the quiver on the vertex
set $Q_0 = [n+1]$ with arrows $i \rightarrow j$ for all $1 \leq i < j \leq n+1$ and if $\theta: Q_0 \rightarrow \mathbb{R}$ denotes 
the weight function defined by $\theta(i) = a_i$ for $1 \leq i \leq n$ and $\theta(n+1) = -a_1 - \cdots - a_n$, then 
the Tesler polytope $\TP_n(\aaa)$ is precisely the polytope $\Delta(\theta)$ considered in 
\cite[Theorem 2.2]{hill}.  By the argument in the last paragraph of \cite[Theorem 2.2]{hill} and \cite[Proposition 2.3]{hill},
the genericity condition on $\theta$ in the hypotheses of \cite[Theorem 2.2]{hill} is equivalent to every entry of $\aaa$ being positive.
The conclusion of \cite[Theorem 2.2]{hill} is essentially the same as the special case of Theorem~\ref{face-poset-characterization}
when $\aaa \in (\ZZ_{> 0})^n$.  When some entries of $\aaa$ are zero, in the terminology of \cite{hill} the weight function
$\theta$ lies on a wall, and the results of \cite{hill} do not apply to $\TP_n(\aaa)$.
\end{remark}

\begin{remark}
When $\aaa \in (\ZZ_{> 0})^n$ is a vector of positive integers, the simplicity of $\TP_n(\aaa)$ guaranteed by Theorem~\ref{simple}  
had been observed previously in the context of flow polytopes.   
The condition that every entry in $\aaa$ is positive is equivalent to $\aaa$ lying in the ``nice chamber" defined by Baldoni and Vergne
in \cite[p. 458]{BV2}. In
\cite[p. 798]{BrionV}, Brion and Vergne observe that this condition on $\aaa$ implies the simplicity of 
$\TP_n(\aaa)$.  The simplicity of $\TP_n(\aaa)$ in this case can also be derived from Hille's characterization of the face poset
\cite{hill} using exactly the same argument as in the proof of Theorem~\ref{simple}.
\end{remark}

\section{Volume of the Tesler polytope $\TP_n({\bf 1)}$}
\label{volume-section}

The aim of this section is to prove Theorem \ref{thm:vol} through a
sequence of results. For ease of reading the section is broken down
into several subsections. We start by stating previous results on volumes and Ehrhart polynomials of flow polytopes and then prove specific lemmas regarding $\TP_n({\bf 1})$.

In this section we work in the field of {\em iterated formal Laurent series}
with $m$ variables as discussed by Haglund, Garsia and Xin in \cite[\S 4]{GHX}. We choose a total order of the variables: $x_1,
x_2,\ldots, x_{m}$ to extract {\em iteratively}  coefficients, constant coefficients,
and residues of an element $f({\bf x})$ in this field. We
denote these respectively by  
\[
\CTn{m} f, \quad [{\bf x}^{\bf a}] := [x_m^{a_m}\cdots x_1^{a_1}] f, \quad   \Res_{x_m}
\cdots \Res_{x_1} f.
\]
For more on these iterative coefficient extractions see \cite[\S
2]{GX}.

\subsection{Generating function of $K_{A_n}({\bf a'})$ and the Lidskii formulas}

Recall that by Lemmas~\ref{lem:flow} and \ref{lem:tes} we have that the normalized
volume $\vol \TP_n({\bf
  a})$ equals the normalized volume $\vol \FP_n({\bf a})$ and
that the number $T_n({\bf a})$ of Tesler matrices is given by the
Kostant partition function $K_{A_n}({\bf a'})$.  By definition, the
latter is given by the following iterated coefficient extraction.
\begin{equation} \label{gs:kostant}
K_{A_n}({\bf a'})= [{\bf x}^{\bf a'}]
\prod_{1\leq i<j \leq n+1} (1-x_ix_j^{-1})^{-1}.
\end{equation}

In addition, the Kostant partition function is invariant under
reversing the order and sign of the netflow vector.

\begin{proposition} \label{prop:revflow}
\[
K_{A_n}(a_1,a_2,\ldots,a_n,-\sum_{i=1}^n a_i) = K_{A_n}(\sum_{i=1}^n
a_i, -a_n,\ldots,-a_2,-a_1).
\]
\end{proposition}

\begin{proof}
Reversing an (integer) flow on the complete graph $k_n$ gives an
involution between (integer) flows with netflow
$(a_1,a_2,\ldots,a_n,-\sum_{i=1}^n a_i)$ and (integer) flows with
netflow $(\sum_{i=1}^n a_i,-a_n,\ldots,-a_2,-a_1)$.
\end{proof}

Assume that ${\bf a}=(a_1,a_2,\ldots,a_n)$ satisfies $a_i \geq  0$ for
$i=1,\ldots,n$. Then the {\bf Lidskii formulas} \cite[Proposition 34, Theorem 37]{BV2} state that

\begin{equation} \label{eq:vol}
\vol \FP_{n}({\bf a}) = \sum_{{\bf i}}
\binom{\binom{n}{2}}{i_1,i_2,\ldots,i_n} a_1^{i_1}\cdots
a_n^{i_n}\cdot K_{A_{n-1}}(i_1-n+1, i_2-n+2,\ldots,i_n),
\end{equation}
and 
\begin{equation} \label{eq:kost}
K_{A_n}({\bf a'}) = \sum_{\mathclap{{\bf i}}}
\binom{a_1+n-1}{i_1}\binom{a_2+n-2}{i_2} \cdots
\binom{a_{n}}{i_{n}} \cdot  K_{A_{n-1}}(i_1-n+1, i_2-n+2,\ldots,i_n),
\end{equation}
where both sums are over weak compositions ${\bf
  i}=(i_1,i_2,\ldots,i_n)$ of $\binom{n}{2}$ with $n$ parts which we
denote as ${\bf i} \models \binom{n}{2}$,  $\ell({\bf i})=n$.

\begin{example}
The Tesler polytope $\TP_3(1,1,1) \cong \FP_3(1,1,1)$ has
normalized volume $4$ since by \eqref{eq:vol}
\[
\vol \FP_3(1,1,1)=\binom{3}{3,0,0} K_{A_2}(1,-1,0) + \binom{3}{2,1,0} K_{A_2}(0,0,0) +
0 = 1\cdot 1 + 3\cdot 1 = 4.
\]
And this polytope has $T_3(1,1,1)=K_{A_3}(1,1,1,-3)=7$ lattice points
(the seven $3\times 3$
Tesler matrices with hook sums $(1,1,1)$; see Figure~\ref{fig:teslerpolygraph}). Indeed by
\eqref{eq:kost}
\[
K_{A_3}(1,1,1,-3) = \binom{1+2}{3}\binom{1+1}{0} K_{A_2}(1,-1,0) +
\binom{1+2}{2}\binom{1+1}{1}K_{A_2}(0,0,0) = 7.
\]
\end{example}

\begin{example} \cite{BV2}
If one uses \eqref{eq:vol} on the 
Chan-Robbins-Yuen
polytope $\TP_n(\e_1)$ one
obtains 
\[
\vol \TP_n(1,0,\ldots,0) = K_{A_{n-1}}(-{\textstyle \binom{n-1}{2}},-n+2,\ldots,-1,0),
\]
since the only composition ${\bf i}$ that does not vanish is
$i_1=\binom{n}{2}, i_2=0,\ldots,i_n=0$. By Proposition~\ref{prop:revflow} this is equivalent to the
first identity in Example~\ref{ex:cry}. 
\end{example}

\subsection{Volume of $\TP_n({\bf 1})$ as a constant term}

In this short section we use \eqref{eq:vol} and the generating series \eqref{gs:kostant}
of Kostant partition functions to write the volume of $\TP_n({\bf 1})$ as an iterated
constant term of a formal Laurent series.

\begin{lemma} \label{lemma:vol_to_CT}
\begin{equation} \label{eq:vol_to_CT}
\vol\TP_n({\bf 1}) = \CTn{n}\,\, (x_1+\cdots + x_n)^{\binom{n}{2}} \prod_{1\leq i<j\leq n}
(x_j-x_i)^{-1},
\end{equation}
where $\CTn{n} f$ denotes the iterated constant term of $f$.
\end{lemma} 

\begin{proof} 
By \eqref{eq:vol} and  Proposition~\ref{prop:revflow} we have that
\begin{align*}
\vol\TP_n({\bf 1})  &= \sum_{{\bf i} \models \binom{n}{2},
  \ell({\bf i})=n} \binom{\binom{n}{2}}{i_1,i_2,\ldots,i_{n}}
\cdot K_{A_{n-1}}(i_1-n+1, i_2-n+2,\ldots,i_n)\\
&= \sum_{{\bf i} \models \binom{n}{2},
  \ell({\bf i})=n} \binom{\binom{n}{2}}{i_1,i_2,\ldots,i_{n}}
\cdot K_{A_{n-1}}(-i_n,1-i_{n-1},2-i_{n-2},\ldots,n-1-i_{1}).
\end{align*}
We use \eqref{gs:kostant} to rewrite this as 
\[
\vol\TP_n({\bf 1}) 
= \sum_{{\bf i} \models \binom{n}{2},
  \ell({\bf i})=n} 
\binom{\binom{n}{2}}{i_1,i_2,\ldots,i_{n}}  [{\bf
  x}^{\bf \delta_n - {\bf i}}] \prod_{1\leq i<j\leq n}
(1-x_ix_j^{-1})^{-1},
\]
where $\delta_n = (0,1,2\ldots,n-1)$. Since $[{\bf x}^{\bf a}]
f = \CTn{n} {\bf x}^{-{\bf a}} f$ then
\[
\vol\TP_n({\bf 1})  = \CTn{n} \sum_{{\bf i} \models \binom{n}{2},
  \ell({\bf i})=n} {\bf x}^{{\bf i}-\delta_n}
\binom{\binom{n}{2}}{i_1,i_2,\ldots,i_{n}} 
\prod_{1\leq i<j\leq n}
(1-x_ix_j^{-1})^{-1}.
\]
Using $\displaystyle \prod_{1\leq i<j\leq n}
(1-x_ix_j^{-1})^{-1} = {\bf x}^{\delta_n} \prod_{1\leq i<j\leq n}
(x_j-x_i)^{-1}$ we get
\[
\vol\TP_n({\bf 1}) = \CTn{n} \prod_{1\leq i<j\leq n}
(x_j-x_i)^{-1} \sum_{{\bf i} \models \binom{n}{2},
  \ell({\bf i})=n} \binom{\binom{n}{2}}{i_1,i_2,\ldots,i_{n}} {\bf
  x}^{\bf i}.
\]
An application of the multinomial theorem yields the desired result.
\end{proof}

\subsection{A Morris-type constant term identity}

Let $e_k=e_k(x_1,x_2,\ldots,x_n)$
denote the $k\textsuperscript{th}$ elementary symmetric polynomial. In particular
$e_1=x_1+x_2+\cdots +x_n$. For $n\geq 2$ and nonnegative integers
$a,c$ we define $L_n(a,c)$ to be the following iterated constant term:
\begin{equation} \label{eq:ct}
L_n(a,c) := \CTn{n} \,\, e_1^{(a-1)n+ c\binom{n}{2}}  \prod_{i=1}^{n} x_i^{-a+1} \prod_{1\leq
  i < j \leq n} (x_i-x_j)^{-c}.
\end{equation}

Note that by Lemma~\ref{lemma:vol_to_CT} we have that
\begin{equation} \label{eq:volctid}
\vol \TP_n({\bf 1}) =
L_n(1,1).
\end{equation}

Next we give a product formula for $L_n(a,c)$ that for
$a=c=1$ yields \eqref{eq:volformula}. We postpone the proof to the
next section.

\begin{lemma} \label{main:lemma}
For $n\geq 2$ and nonnegative integers
$a,c$ we have that
% $L_n(a,c):=\CTn{n} \,\, e_1^{(a-1)n+ c\binom{n}{2}} \prod_{i=1}^{n} x_i^{-a+1} \prod_{1\leq
%  i < j \leq n} (x_i-x_j)^{-c}$ equals
\begin{equation}
\label{eq:ctterm}
L_n(a,c) = \big((a-1)n+c{\textstyle\binom{n}{2}}\big)!
\prod_{i=0}^{n-1} \frac{\Gamma(1+c/2)}{\Gamma(1+(i+1)c/2)\Gamma(a+ic/2)},
\end{equation}
where $\Gamma(\cdot)$ is the Gamma function.
\end{lemma}

\begin{corollary} \label{cor:case11}
\begin{equation}
L_n(1,1) = \frac{\binom{n}{2}!\cdot  2^{\binom{n}{2}}}{\prod_{i=1}^n
  i!}. \label{eq:case11}
\end{equation}
\end{corollary}

\begin{proof}
Set $a=1$ and $c=1$ in \eqref{eq:ctterm} and obtain
\[
L_n(1,1) = \binom{n}{2}! \prod_{i=0}^{n-1} \frac{\Gamma(3/2)}{\Gamma(1+(i+1)/2)\Gamma(1+i/2)},
\]
since $\Gamma(3/2) = \sqrt{\pi}/2$ and by the duplication formula of
$\Gamma(\cdot)$ this becomes
\[
L_n(1,1) = \binom{n}{2}! \prod_{i=0}^{n-1} \frac{2^{i}}{(i+1)!} = \frac{\binom{n}{2}!\cdot  2^{\binom{n}{2}}}{\prod_{i=1}^n
  i!},
\]
as desired. 
\end{proof}

Drew Armstrong (private communication) noted the resemblance of the
product in the  RHS \eqref{eq:case11} with the number of standard Young
tableaux of staircase shape. Indeed, if we let  $f^{(n-1,n-2,\ldots,1)}$ be the number of
standard Young tableaux of shape $(n-1,n-2,\ldots,1)$ which by the
hook-length formula equals
\[
f^{(n-1,n-2,\ldots,1)} = \frac{\binom{n}{2}!}{\prod_{k=1}^{n-1}
                             (2k-1)^{n-k}},
\]
then one can show that $L_n(1,1)$ is divisible by this number. The
ratio of these numbers is a product of consecutive Catalan numbers.

\begin{proposition} \label{prop:sytcase11} 
\begin{equation} 
\frac{\binom{n}{2}!\cdot  2^{\binom{n}{2}}}{\prod_{i=1}^n
  i!} = f^{(n-1,n-2,\ldots,1)} \cdot \prod_{i=1}^{n-1}
Cat(i).
\end{equation}
\end{proposition}

\begin{proof}
The identity is easily verified using the formula for $f^{(n-1,n-2,\ldots,1)}$
and for $Cat(i)=\frac{1}{i+1}\binom{2i}{i}$.
\end{proof}

\begin{remark}
When we set $a=1$ and $c=2$ in \eqref{eq:ctterm} one can also show that 
\begin{equation}
L_n(1,2)  = \frac{(n(n-1))!}{n!(\prod_{i=1}^{n-1} i!)^2} = f^{(n-1)^n}
\cdot \prod_{i=1}^{n-1} \left( \frac{i+1}{2}
           Cat(i)^2\right),
\end{equation}
where $f^{(n-1)^n}$ is the number of standard Young tableaux of
rectangular shape $(n-1)^n$ which equals $(n(n-1))! \prod_{k=0}^{n-1}
k! \,/\, \prod_{k=0}^{n-1} (n+k-1)!$.
We were unable to find  similar identities relating $L_n(1,c)$, $c\geq
3$ with the number of SYT of shape $\lambda$.
\end{remark}

%\begin{remark}
%Consider the constant term of $(1-e_1)^{-1} \prod_{i=1}^{n} x_i^{-a+1} \prod_{1\leq
 % i < j \leq n} (x_i-x_j)^{-c}$. Since  $(1-e_1)^{-1} = \sum_{k\geq 0} e_1^k$ then by linearity of $\CTn{n}$,
%\begin{multline*}
% \CTn{n} \,\, (1-e_1)^{-1} \prod_{i=1}^{n} x_i^{-a+1} \prod_{1\leq
 % i < j \leq n} (x_i-x_j)^{-c} =\\
 %\sum_{k\geq 0} \CTn{n} \,\, e_1^k\prod_{i=1}^{n} x_i^{-a+1} \prod_{1\leq
  %i < j \leq n} (x_i-x_j)^{-c}.
%\end{multline*}
%By degree considerations each constant term is zero if $k\neq (a-1)n+
%c\binom{n}{2}$, thus $L_n(a,c)$ is also the following iterated
%constant term
%\begin{equation}
%\label{eq:ctterm2}
%L_n(a,c) = \CTn{n} \,\, (1-e_1)^{-1} \prod_{i=1}^{n} x_i^{-a+1} \prod_{1\leq
 % i < j \leq n} (x_i-x_j)^{-c}.
%\end{equation}
%\end{remark}

\begin{remark}
A similar iterated constant term identity to \eqref{eq:ctterm}
is Zeilberger's variation of the Morris constant term identity \cite{Z}
used to prove \eqref{eq:z}. We state the version in  \cite[\S 3.5]{GX}: for $n\geq 2$
and nonnegative integers $a,b,c$ let
\begin{equation} \label{eq:MorrisId}
M_n(a,b,c) := \CTn{n} \,\, \prod_{i=1}^{n} x_i^{-a+1} (1-x_i)^{-b} \prod_{1\leq
  i < j \leq n} (x_i-x_j)^{-c} 
\end{equation}
then
\[
M_n(a,b,c) = \prod_{j=0}^{n-1} \frac{\Gamma(1+c/2)\Gamma(a+b-1+(n+j-1)c/2)}{\Gamma(1+(j+1)c/2)\Gamma(a+jc/2)\Gamma(b+jc/2)},
\]
and in particular
\[
M_n(1,1,1) = \prod_{i=0}^{n-1} \frac{1}{i+1}\binom{2i}{i}.
\]

Moreover, let $h_k(x_1,\ldots,x_n)$ denote the $k\textsuperscript{th}$ complete symmetric
polynomial in the variables $x_1,\ldots,x_n$. Since $\prod_{i=1}^n (1-x_i)^{-1} = \sum_{k\geq 0}
h_k(x_1,\ldots,x_n)$ then by linearity of  $\CTn{n}$ and degree considerations,
$M_n(a,1,c)$ can be expressed as a sum of iterated constant term extractions
all except one are zero. Thus
\begin{equation} \label{eq:altMorris}
M_n(a,1,c) = \CTn{n} \,\,  h_{((a-1)n+c
  \binom{n}{2})}(x_1,\ldots,x_n) \prod_{i=1}^{n} x_i^{-a+1}\prod_{1\leq
  i < j \leq n} (x_i-x_j)^{-c}. 
\end{equation}
This alternate description of $M_n(a,1,c)$ resembles the original 
definition of $L_n(a,c)$ in \eqref{eq:ct}.  Conversely, one can show
using $(1-e_1)^{-1}  = \sum_{k\geq 0} e_1^k$, linearity, and degree
considerations that $L_n(a,c)$ equals the following iterated constant term
\begin{equation} \label{eq:altctid}
L_n(a,c) = \CTn{n} \,\, (1-e_1)^{-1} \prod_{i=1}^{n} x_i^{-a+1} \prod_{1\leq
  i < j \leq n} (x_i-x_j)^{-c},
\end{equation}
which resembles the original description of $M_n(a,1,c)$.
\end{remark}

\subsection{Proof of Lemma~\ref{main:lemma} via Baldoni-Vergne recurrence
approach}

To prove Lemma~\ref{main:lemma} we follow Xin's \cite[\S 3.5]{GX} simplified recursion approach of the proof by
Baldoni-Vergne \cite{BVMorris} of the Morris identity
\eqref{eq:MorrisId}. 

\smallskip

\noindent {\bf Outline of the proof:} First, for nonnegative integers  $n \geq 2,
a,c$ and $\ell=0,\ldots,n$ we introduce the constants
\[
C_n(\ell, a, c) := \CTn{n} \frac{P_{\ell} \cdot e_1(x_1,\ldots,x_n)^{(a-1)n+ c\binom{n}{2}-\ell}}{\prod_{i=1}^n  x_i^{a-1} \prod_{i=1}^n(x_i-x_j)^c},
\]
where $P_{\ell} = \ell! (n-\ell)!
  e_{\ell}(x_1,\ldots,x_n)$. Note that $C_n(0,a,c) = n!L_n(a,c)$. Second,
  we show that $C_n(\ell,a,c)$ satisfy certain linear relations (Proposition~\ref{prop:rels}). Third, we show that these
  relations uniquely determine the constants $C_n(\ell,a,c)$
  (Proposition~\ref{prop:recunique}). Lastly, in Proposition~\ref{prop:eqCn} we define 
  $C'_n(\ell,a,c)$ as certain products of Gamma functions such that
  $C'_n(0,a,c)/n!$ coincides with the expression on the right-hand-side of
  \eqref{eq:ctterm}. We then show that $C'_n(\ell,a,c)$ satisfy the same
  relations as $C_n(\ell,a,c)$ and since these relations determine
  uniquely the constants then
  $C'_n(\ell,a,c)=C_n(\ell,a,c)$. This completes the proof of the Lemma.  

\smallskip

The  $C_n(\ell,a,c)$ satisfy the
following relations.

\begin{proposition} \label{prop:rels}
Let $C_n(\ell,a,c)$ be defined as above then for $1\leq \ell \leq n$
we have:
\begin{align}
\frac{C_n(\ell,a,c)}{C_n(\ell-1,a,c)} &= \frac{a-1+c(n-\ell)/2}{(a-1)n + c\binom{n}{2} - \ell +1}, \label{rec5}\\
C_n(n,a,c) &= C_n(0,a-1,c), \label{rec1}\\
C_n(n-1,1,c) & = C_{n-1}(0,c,c), \label{rec2} \qquad  (\text{if } n>1)\\
C_n(0,1,0) & = n!, \label{rec3}\\
C_n(\ell,0,c) & = 0. \label{rec4} 
\end{align}
\end{proposition}

\begin{proof}
The relations \eqref{rec1}-\eqref{rec4} follow from the same proof as
in  \cite[Theorem 3.5.2]{GX}
$C_n(\ell,a,c)$. 

We now prove \eqref{rec5}. Let $U_{\ell}=e_1^{(a-1)n+ c\binom{n}{2}-\ell}/(\prod_{i=1}^n  x_i^{a}
\prod_{i=1}^n(x_i-x_j)^c)$, since $\CT_y g(y) = \Res_y y g(y)$ then 
\begin{equation} \label{eq:ct2res}
C_n(\ell, a, c) = \Res_{x_n} \cdots \Res_{x_1} \, P_{\ell}
U_{\ell},
\end{equation}
Next we calculate the
following derivative with respect to $x_1$.
\begin{multline}
\frac{\partial}{\partial{x_1}}  e_1\cdot  x_1x_2\cdots x_{\ell}
U_{\ell} = \left((a-1)n+c{\textstyle \binom{n}{2}}-\ell+1\right)x_1\cdots x_{\ell}
U_{\ell} +  (1-a)x_2\cdots x_{\ell} U_{\ell-1} +\\
 - c \cdot 
x_1\cdots x_{\ell} \sum_{j=2}^n \frac{U_{\ell-1}}{x_1-x_j}. \label{eq:master} 
\end{multline}

If $c$ is odd then $U_{\ell}$ is anti-symmetric. If we
anti-symmetrize \eqref{eq:master} over the symmetric group $\mathfrak{S}_n$, we get
\begin{multline*}
\sum_{w\in \mathfrak{S}_n} (-1)^{inv(w)} w \cdot \left(\frac{\partial}{\partial{x_1}}  e_1\cdot x_1x_2\cdots x_{\ell}
U_{\ell} \right) =\\
\left((a-1)n+c{\textstyle \binom{n}{2}}-\ell+1\right)P_{\ell}U_{\ell} +  (1-a)
P_{\ell-1} U_{\ell-1} - c\sum_{w\in \mathfrak{S}_n} w\cdot
x_1\cdots x_{\ell} \sum_{j=2}^n \frac{U_{\ell-1}}{x_1-x_j}
\end{multline*}
One can check that 
\[
2 \sum_{w\in \mathfrak{S}_n} w\cdot
x_1\cdots x_{\ell} \sum_{j=2}^n \frac{1}{x_1-x_j} =
(n-\ell) P_{\ell-1}. 
\]
So putting everything together for $c$ odd we obtain 
\begin{multline}
\sum_{w\in \mathfrak{S}_n} (-1)^{inv(w)} w \cdot
\left(\frac{\partial}{\partial{x_1}}  e_1 \cdot x_1x_2\cdots x_{\ell}
U_{\ell} \right) = \\
\left((a-1)n+c{\textstyle \binom{n}{2}}-\ell+1\right)P_{\ell}U_{\ell} - (a-1 +
c(n-\ell)/2)P_{\ell-1}U_{\ell-1}. \label{eq:odd} 
\end{multline}
Next, if $c$ is even, $U_{\ell}$ is symmetric. If we
symmetrize \eqref{eq:master} over $\mathfrak{S}_n$ and do similar
simplifications as in the previous case we get
\begin{align}
\sum_{w\in \mathfrak{S}_n} &  w \cdot \left(\frac{\partial}{\partial{x_1}}  e_1 x_1x_2\cdots x_{\ell}
U_{\ell} \right) = \label{eq:even}\\
& \left((a-1)n+c{\textstyle \binom{n}{2}}-\ell+1\right)P_{\ell}U_{\ell} - (a-1 +
c(n-\ell)/2)P_{\ell-1}U_{\ell-1}. \nonumber
\end{align}
Finally, we take the iterated residue $\Res_{x_n}\cdots \Res_{x_1}$ of \eqref{eq:odd} and
\eqref{eq:even}. Since the left-hand-side of these two equations consist of
sums of derivatives with respect to $x_1,\ldots,x_n$, then their iterated residues
$\Res_{\bf x}$
are zero \cite[Remark 3(c), p. 15]{BVMorris}. This combined with
\eqref{eq:ct2res} yields
\[
0 = \left((a-1)n+c{\textstyle \binom{n}{2}}-\ell+1\right) C_n(\ell,a,c) - (a-1+c(n-\ell)/2) C_n(\ell-1,a,c),
\]
which proves \eqref{rec5} for $c$ even or odd.
\end{proof}

We now show that the recurrences \eqref{rec5}-\eqref{rec4} determine entirely the
constants $C_n(\ell,a,c)$ (same algorithm as in
\cite[p. 10]{BVMorris}). 

\begin{proposition} \cite[p. 10]{BVMorris} \label{prop:recunique}
The recurrences \eqref{rec5}-\eqref{rec4} determine uniquely the
constants $C_n(\ell,a,c)$.
\end{proposition}

\begin{proof}
We give an algorithm to compute the constants $C_n(\ell,a,c)$
recursively using \eqref{rec5}-\eqref{rec4}. The algorithm has the
following three cases:

\medskip

\noindent {\em Case 1.}
If $c=0$ and $a>1$ we use \eqref{rec5} repeatedly to increase $\ell$ up to
$n$. We can use this recursion since $a-1+c(n-\ell)=a-1>0$. If
$\ell=n$ then we can apply \eqref{rec1} and go from $C_n(n,a,0)$ to
$C_n(0,a-1,0)$:
\[
\xymatrix{
C_n(\ell,a,0)  \ar[r]^(.45){\eqref{rec5}} & C_n(\ell+1,a,0) 
\ar@{-->}[rr]^{\eqref{rec5}^*}
&& C_n(n,a,0)
\ar[r]^(.45){\eqref{rec1}} & C_n(0,a-1,0).
}
\]
Thus computing $C_n(\ell,a,0)$ reduces to finding $C_n(0,1,0)$ which
equals $n!$ by \eqref{rec3}. 

\medskip

\noindent {\em Case 2.} If $c>0$ and $a>1$ we use \eqref{rec5}
repeatedly to increase $\ell$ up
  to $n$. We can use this recursion since
$a-1+c(n-\ell)=a-1>0$. If $\ell=n$ then we apply \eqref{rec1} and go
from $C_n(n,a,c)$ to $C_n(0,a-1,c)$: 
\[
\xymatrix{
C_n(\ell,a,c)  \ar[r]^(.45){\eqref{rec5}} & C_n(\ell+1,a,c)
\ar@{-->}[rr]^{\eqref{rec5}^*} && C_n(n,a,c)
\ar[r]^(.45){\eqref{rec1}} & C_n(0,a-1,c).
}
\]
Thus computing $C_n(\ell,a,c)$ reduces to finding $C_n(0,1,c)$. 

\medskip

\noindent {\em Case 3.}
To compute $C_n(0,1,c)$ with $c>0$, we use \eqref{rec5} repeatedly to increase
$\ell$ from $0$ up
to $n-1$. Then we can apply \eqref{rec2} and go from $C_n(n-1,1,c)$ to $C_{n-1}(0,c,c)$:
\[
\xymatrix{
C_n(0,1,c)  \ar[r]^{\eqref{rec5}} & C_n(1,1,c)
\ar@{-->}[rr]^(.45){\eqref{rec5}^*} &&   C_n(n-1,1,c)
\ar[r]^{\eqref{rec2}} & C_{n-1}(0,c,c).
}
\]
Thus by iterating this reduction with Case 2 we see that computing $C_n(0,1,c)$ reduces to finding $C_1(\ell,a,c)$. Having
$n=1$ guarantees there is no term 
\[
\prod_{1\leq
  i<j\leq n} (x_i-x_j)^{-c}.
\]
 So
$C_1(\ell,a,c)=C_1(\ell,a,0)$ which we can compute with Case 1. 
\end{proof}

Next we give an explicit
product formula for $C_n(\ell,a,c)$. We prove this by showing that the
formula satisfies relations~\eqref{rec5}-\eqref{rec4}  which by
Proposition~\ref{prop:recunique} determine uniquely  $C_n(\ell,a,c)$.

\begin{proposition} \label{prop:eqCn}
If $c>0$ or if $a>1$ then for $1\leq \ell \leq n$ then
\begin{equation} \label{eq:form1}
C_n(\ell,a,c) =  C_n(0,a,c)\prod_{j=1}^{\ell} \frac{a-1+(n-j)c/2}{(a-1)n + c{n \choose
  2} - j+1}.    
\end{equation}
if $a\geq 1$ then
\begin{equation} \label{eq:form2}
C_n(0,a,c) = n!\cdot\Gamma\big(1+(a-1)n+c{\textstyle\binom{n}{2}}\big)
\prod_{i=0}^{n-1} \frac{\Gamma(1+c/2)}{\Gamma(1+(i+1)c/2)\Gamma(a+ic/2)}.
\end{equation}
\end{proposition}

\begin{proof}
By Proposition~\ref{prop:recunique} it suffices to check that the formulas for $C_n(\ell,a,c)$ and $C_n(0,a,c)$ in
\eqref{eq:form1}, \eqref{eq:form2} satisfy the
relations \eqref{rec5}-\eqref{rec4}.

Let $C'_n(\ell,a,c)$ and $C'_n(0,a,c)$ be
the formulas in the right-hand-side of \eqref{eq:form1} and \eqref{eq:form2}
respectively.

Relation \eqref{rec5} is apparent from the definition of
$C'_n(\ell,a,c)$. 

Next we check that $C'_n(\ell,a,c)$ satisfies \eqref{rec1}. Using $\Gamma(t+1)=t\Gamma(t)$ repeatedly
we obtain:
\begin{align*}
\MoveEqLeft \frac{C'_n(n-1,a,c)}{C'_n(0,a-1,c)} = \\
&=\frac{\Gamma(1+(a-1)n +c{\textstyle \binom{n}{2}})}{\Gamma(1+(a-2)n
  + c{\textstyle \binom{n}{2}})} \prod_{j=1}^n
\frac{a-1+(n-j)c/2}{(a-1)n+c{\textstyle \binom{n}{2}}-j+1}
\prod_{i=0}^n
\frac{\Gamma(a-1+ic/2)}{\Gamma(a+ic/2)}\\
&= \prod_{j=1}^n ((a-1)n + c{\textstyle \binom{n}{2}} -j + 1) \prod_{j=1}^n
\frac{a-1+(n-j)c/2}{(a-1)n+c{\textstyle \binom{n}{2}}-j+1}
\prod_{i=0}^n
\frac{1}{a-1+ic/2}\\& =1,
\end{align*}
as desired. 

Next we verify \eqref{rec2}. Again, using
$\Gamma(t+1)=t\Gamma(t)$ repeatedly we obtain:
\begin{align*}
\MoveEqLeft \frac{C'_n(n-1,1,c)}{C'_{n-1}(0,c,c)} =\\
&= 
\frac{\prod_{j=1}^{n-1} (n-j)c/2}{\prod_{j=1}^{n-1} c{\textstyle \binom{n}{2}}-j+1}
\frac{n\Gamma(1+c{\textstyle \binom{n}{2}})}{\Gamma(1+c{\textstyle
    \binom{n}{2}}-(n-1))} \times \\
& \quad \times \frac{\Gamma(1+c/2)
}{\Gamma(1+(n-1)c/2)\Gamma(1+nc/2)} \frac{\prod_{i=0}^{n-2}
  \Gamma(c(i+2)/2)}{\prod_{i=0}^{n-2}
  \Gamma(1+ic/2)}\\
&=\frac{\prod_{j=1}^{n-1} (n-j)c/2}{\prod_{j=1}^{n-1} c{\textstyle \binom{n}{2}}-j+1}\frac{n
  \prod_{j=1}^{n-1} c{\textstyle \binom{n}{2}}-j+1}{1} 
\prod_{j=2}^n \frac{\Gamma(jc/2)}{\Gamma(1+jc/2)}\\
&= n\prod_{j=1}^{n-1} (n-j)c/2 
\prod_{j=2}^n \frac{1}{jc/2} = 1,
\end{align*}
as desired.

Finally, it is trivial to check that $C'_n(\ell,a,c)$ satisfy
\eqref{rec3} and \eqref{rec4}. Thus since $C'_n(\ell,a,c)$ satisfy
relations \eqref{rec5}-\eqref{rec4} and by Proposition~\ref{prop:recunique}  these
relations uniquely determine the constants $C_n(\ell,a,c)$ then $C'_n(\ell,a,c)=C_n(\ell,a,c)$.
\end{proof}

To conclude, since $C_n(0,a,c) = n!\cdot L_n(a,c)$ then
Lemma~\ref{main:lemma} follows from \eqref{eq:form2} in
Proposition~\ref{prop:eqCn}. By Corollary~\ref{cor:case11} and
Proposition~\ref{prop:sytcase11} $L_n(1,1)$ yields the desired
formula for the volume of $\TP_n({\bf 1)}$ which completes the proof of Theorem \ref{thm:vol}.

\section{Final remarks} \label{final-remarks}

\subsection{Diagonal harmonics and polytopes}

 Example~\ref{ex:tesler} states Haglund's result from \cite{Hag} showing that the
bigraded Hilbert series of the space $DH_n$ is given by a weighted sum
over Tesler matrices in $\mathcal{T}_n(1,1,\ldots,1)$. The space $DH_n$ has dimension $(n+1)^{n-1}$, the number of parking
functions of size $n$. A conjecture of Haglund and Loehr \cite{HL}, 
settled by Carlsson and Mellit \cite{CM} with their proof of the more
general {\em shuffle conjecture} \cite{HHLRU}, expresses the LHS as
\begin{equation} \label{eq:HL}
\mathcal{H}(DH_n,q,t)  = \sum_{\pi} q^{\dinv(\pi)}t^{\area(\pi)},
\end{equation}
where the sum is over parking functions $\pi$. For definitions of
the statistics $\dinv$ and $\area$ see \cite{Hag2}. By definition
$\mathcal{H}(DH_n,q,t)$ is a polynomial in $\mathbb{N}[q,t]$ and symmetric
in $q$ and $t$. The right-hand sides of \eqref{eq:HL} and \eqref{eq:haglund} 
give different combinatorial models for this Hilbert series where the ($q,t$
positivity, $q,t$ symmetry) are (trivial, non-trivial) and
(non-trivial, trivial) respectively. It remains open to prove directly
the equality of these models:
\begin{equation} \label{pf2tesler}
\sum_{\pi} q^{\dinv(\pi)}t^{\area(\pi)} = \sum_{A \in \mathcal{T}_n(1,1,\ldots,1)} wt(A),
\end{equation}
for $wt(A)$ as defined in \eqref{eq:Hwt}. Levande \cite{PL} verified this
identity for $(q,0)$ and $(1,t)$.  In particular, when $q=1,t=1$,
$wt(A) \mid_{q=1,t=1} = 0$ for any $n\times n$ Tesler matrix $A$
with more than $n$ nonzero entries and the matrices that survive
are the permutation Tesler matrices each with
 $n$ nonzero entries. Thus \eqref{pf2tesler} at  $q=1,t=1$ becomes
\[
(n+1)^{n-1} = \sum_{A} \prod_{i,j \,:\, a_{ij} >0} a_{ij} 
\] 
where the sum is over the $n!$ permutation Tesler matrices in
$\mathcal{T}_n({\bf 1})$; the vertices of polytope  $\TP_n({\bf
  a})$. This curious identity was proved combinatorially in
\cite[\S 5]{AGHRS} extending a function from Levande \cite{PL} from Tesler
matrices to permutations.

Analogously, an important subspace of the space $DH_n$ is the alternant
$DH_n^{\varepsilon}$ that has dimension
$Cat(n)=\frac{1}{n+1}\binom{2n}{n}$. The bigraded Hilbert series of
$DH^{\varepsilon}_n$ has the following combinatorial model by
Garsia and Haglund \cite{GH1,GH2}
\begin{equation} \label{eq:qtCat}
\mathcal{H}(DH^{\varepsilon}_n,q,t) = \sum_{P} q^{\area(P)}t^{\bounce(P)},
\end{equation}
where the sum is over {\em Dyck paths} $P$ of size $n$, see \cite[\S
3]{Hag2} for the definition of $\bounce$. Gorsky and
Negut \cite{GoN} also expressed this Hilbert series as a weighted sum
over Tesler matrices:
\begin{equation} \label{eq:qtCatGN}
\mathcal{H}(DH^{\varepsilon}_n,q,t) = \sum_{A \in \mathcal{T}_n(1,1,\ldots,1)} wt'(A),
\end{equation}
where
\[
wt'(A) = \prod_{a_{ii+1}>0} ( [a_{ii+1}+1]_{q,t} - [a_{ii+1}]_{q,t} )
\prod_{j>i+1 \,:\, a_{ij}>0} (-M)[a_{ij}]_{q,t},
\] 
for $M= (1-q)(1-t)$ and $[b]_{q,t} = (q^b-t^b)/(q-t)$ as in
\eqref{eq:Hwt}. When we set $q=1,t=1$ in  \eqref{eq:qtCatGN}, by the definition of $wt'(A)$ only the Tesler
matrices $A$ with support in the diagonals $a_{ii}$ and $a_{ii+1}$
survive each with weight $1$. So \eqref{eq:qtCatGN} becomes
\begin{equation} \label{eq:PS}
\frac{1}{n+1}\binom{2n}{n} = \#\{ A \in \mathcal{T}_n(1,1,\ldots,1) \,:\,
a_{ij}=0, j>i+1\}.
\end{equation}
This identity can be proved in the context of flow
polytopes. Namely, translating from flow polytopes (see Lemma~\ref{lem:flow}) Baldoni-Vergne \cite{BV2} noticed that the polytope
\[
\{ (m_{i,j})\in \TP_n({\bf a}) \,:\, m_{i,j}=0, j
> i+1\}
\]
is the {\em Pitman-Stanley} polytope \cite{PSt} and when ${\bf a} = {\bf 1}$, this polytope has
  $\frac{1}{n+1}\binom{2n}{n}$ lattice points, explaining
  \eqref{eq:PS}, and volume $(n+1)^{n-1}$ (see \cite[\S1, \S5]{PSt}).

\subsection{Enumeration of Tesler matrices}

There is no known explicit formula for the number $T_n({\bf 1})$ of Tesler
matrices of size $n$. More than 20 terms of the sequence $\{T_n({\bf
  1})\}_{n=1}$ have been computed in the OEIS \cite[\href{http://oeis.org/A008608}{A008608}]{OEIS}:
\begin{quote}
1, 2, 7, 40, 357, 4820, 96030, 2766572, 113300265, 6499477726,
\linebreak
515564231770, 55908184737696, \ldots
\end{quote}

Regarding asymptotic of this sequence we give some preliminary lower
and upper bounds that follows from a recursive construction by
Drew Armstrong \cite{A}. 

\begin{proposition} \label{prop:bounds} $n! \leq T_n({\bf 1}) \leq 2^{\binom{n}{2}}$.
\end{proposition}

\begin{proof}
Let $\pi: \mathcal{T}_n(a_1,\ldots,a_{n-1},a_n) \to
\mathcal{T}_{n-1}(a_1,\ldots,a_{n-1})$ defined by $\pi: (a_{i,j}) \mapsto (b_{i,j})$
where $b_{i,j} = \begin{cases} a_{i,i} + a_{i,n} &\text{ if } i=j,\\
  a_{i,j} &\text{ if } i\neq j \end{cases}$. See Figure~\ref{fig:DrewProj} for an
example of $\pi$. The
map $\pi$ is surjective and for each $B\in
\mathcal{T}_{n-1}(a_1,\ldots,a_{n-1})$,  the size of the preimage is $\pi^{-1}(B) = \prod_{i=1}^{n-1}
(1+b_{i,i})$. Thus
\begin{equation} \label{eq:ptnKPF}
T_n(a_1,\ldots,a_n) =\sum_{B \in
  \mathcal{T}(a_1,\ldots,a_{n-1})} \prod_{i=1}^{n-1}
(1+b_{i,i}).
\end{equation}
For the case ${\bf a} = {\bf 1}$ one can show that if $B \in
\mathcal{T}_{n-1}({\bf 1})$ then  $n\leq \pi^{-1}(B) \leq 2^{n-1}$.
Using these bounds for $\pi^{-1}(B)$ in \eqref{eq:ptnKPF} yields
\[
n\cdot T_{n-1}({\bf 1}) \leq T_n({\bf 1}) \leq 2^{n-1}\cdot T_{n-1}({\bf 1}).
\] 
Iterating these bounds give the desired result. 

An alternative proof of the lower bound is as follows: the matrices in $\mathcal{T}_n({\bf 1})$
include the $n!$ permutation Tesler matrices of size $n$. 
\end{proof}

\begin{figure}
\begin{center}
\includegraphics{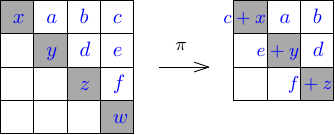}
\caption{Illustration of the projection $\pi$ used in the proof of
  Proposition~\ref{prop:bounds}.}
\label{fig:DrewProj}
\end{center}
\end{figure}

\subsection{Combinatorial proof volume of CRY and Tesler polytopes}

The product formulas \eqref{eq:z} and \eqref{eq:volformula}  for the volumes of the CRY 
and the Tesler
polytopes involving Catalan numbers and number of SYT
suggest a combinatorial proof that has been elusive since Zeilberger's
proof of \eqref{eq:z}. The current proofs of the formulas use the
Lidskii formula \eqref{eq:vol} for the volume of flow polytopes to
translate the problem to evaluations of Kostant partition functions
via constant term identities. 

It is also not clear why the volume of the CRY polytope divides the
volume of the Tesler polytope in terms of operations on polytopes. Curiously, using constant term identities it
  is possible to express the volume of the Tesler polytope as a
  nonnegative sum of terms two of which are $f^{(n-1,n-2,\ldots,1)}$
  and $\prod_{i=0}^{n-1} Cat(i)$. Namely, by \eqref{eq:vol_to_CT} the volume of the Tesler polytope
$\TP_n({\bf 1})$ is the constant term of  $(e_1(x_1,\ldots,x_n))^{\binom{n}{2}}
  \prod_{1\leq i<j \leq n} (x_i-x_j)^{-1}$. Since $e_1^{\binom{n}{2}} = \sum_{\lambda
  \vdash \binom{n}{2}} f^{\lambda} s_{\lambda}$ where $s_{\lambda}$ is
the Schur function of $\lambda$, then by linearity of $\CTn{n}$ 
\[
f^{(n-1,n-2,\ldots,1)} \prod_{i=0}^{n-1} Cat(i) = \sum_{\lambda\vdash \binom{n}{2}}
f^{\lambda} \CTn{n}
s_{\lambda}(x_1,\ldots,x_n)  \prod_{1\leq i<j \leq n} (x_i-x_j)^{-1}.
\]
First, when
$\lambda=(n-1,n-2,\ldots,1)$ then we get $f^{(n-1,n-2,\ldots,1)}$ and by degree considerations and
\eqref{gs:kostant} one can show that
\begin{multline*}
\CTn{n}
s_{(n-1,n-2,\ldots,1)} \prod_{1\leq i<j \leq n}
(x_i-x_j)^{-1} =\\
\CTn{n} x_2x_3^2\cdots
x_n^{n-1} \prod_{1\leq i<j \leq n}
(x_i-x_j)^{-1} =  K_{A_{n-1}}({\bf 0}) = 1.
\end{multline*}
Second, when $\lambda=(\binom{n}{2})$ then $f^{(\binom{n}{2})}=1$ and by
the version \eqref{eq:altMorris} of the Morris identity
we get
\[
\CTn{n} s_{(\binom{n}{2})}\prod_{1\leq i<j \leq n}
(x_i-x_j)^{-1} = M_n(1,1,1)=\prod_{i=0}^{n-1}Cat(i).
\]
This of course this still leaves the question of why the nonnegative
sum ends up being the product $f^{(n-1,n-2,\ldots,1)}\cdot
\prod_{i=0}^{n-1}Cat(i)$ unanswered.
\bigskip

\noindent {\bf Acknowledgements:} 
We thank Drew Armstrong for many inspiring conversations throughout
this project.
We thank Fran\c{c}ois Bergeron
for suggesting that flow polytopes were related to Tesler matrices, and Ole Warnaar for showing us
simplifications of Gamma functions that led to the compact expression
on the right-hand-side of \eqref{eq:ctterm} from a more complicated
precursor. We also thank Yonggyu Lee for finding an error in a
previous version of the proof of Lemma~\ref{vertex-characterization}.

\end{document}